\documentclass[reqno,10pt,centertags]{amsart}
\usepackage{amsmath,amsthm,amscd,amssymb,latexsym,esint,upref,stmaryrd,
enumerate,color,verbatim,yfonts}
\usepackage{hyperref}

\newcommand*{\mailto}[1]{\href{mailto:#1}{\nolinkurl{#1}}}
\newcommand{\arxiv}[1]{\href{http://arxiv.org/abs/#1}{arXiv:#1}}

\newcommand{\bbC}{{\mathbb{C}}}

\newcommand{\bbN}{{\mathbb{N}}}

\newcommand{\bbR}{{\mathbb{R}}}

\newcommand{\bsH}{{\boldsymbol{H}}}
\newcommand{\bsI}{{\boldsymbol{I}}}
\newcommand{\bsJ}{{\boldsymbol{J}}}
\newcommand{\bsK}{{\boldsymbol{K}}}
\newcommand{\bsL}{{\boldsymbol{L}}}

\newcommand{\bsS}{{\boldsymbol{S}}}
\newcommand{\bsT}{{\boldsymbol{T}}}
\newcommand{\bsU}{{\boldsymbol{U}}}
\newcommand{\bsV}{{\boldsymbol{V}}}

\newcommand{\bsu}{{\boldsymbol{u}}}
\newcommand{\bsv}{{\boldsymbol{v}}}

\newcommand{\cB}{{\mathcal B}}

\newcommand{\cF}{{\mathcal F}}

\newcommand{\cH}{{\mathcal H}}

\newcommand{\cJ}{{\mathcal J}}
\newcommand{\cK}{{\mathcal K}}

\DeclareMathOperator{\supp}{supp}

\DeclareMathOperator{\dom}{dom}

\DeclareMathOperator{\tr}{tr}
\DeclareMathOperator{\spr}{spr}

\DeclareMathOperator*{\nlim}{n-lim}

\renewcommand{\Im}{\text{\rm Im}}

\newcommand{\loc}{\text{\rm{loc}}}

\newcommand{\no}{\notag}
\newcommand{\lb}{\label}
\newcommand{\f}{\frac}

\newcommand{\ol}{\overline}

\newcommand{\hatt}{\widehat} 
\newcommand{\bi}{\bibitem}

\renewcommand{\ge}{\geqslant}

\let\geq\geqslant
\let\leq\leqslant

\makeatletter
\def\theequation{\@arabic\c@equation}

\allowdisplaybreaks 
\numberwithin{equation}{section}

\newtheorem{theorem}{Theorem}[section]

\newtheorem{lemma}[theorem]{Lemma}
\newtheorem{corollary}[theorem]{Corollary}

\newtheorem{hypothesis}[theorem]{Hypothesis}

\theoremstyle{remark}
\newtheorem{remark}[theorem]{Remark}

\begin{document}

\numberwithin{equation}{section}
\allowdisplaybreaks

\title[Reduction of Fredholm Determinants for Semi-Separable Operators]{A Jost--Pais-Type Reduction of (Modified) Fredholm Determinants for Semi-Separable Operators in Infinite Dimensions}

\author[F.\ Gesztesy]{Fritz Gesztesy}  
\address{Department of Mathematics,
University of Missouri, Columbia, MO 65211, USA}
\email{\mailto{gesztesyf@missouri.edu}}
\urladdr{\url{http://www.math.missouri.edu/personnel/faculty/gesztesyf.html}}

\author[R.\ Nichols]{Roger Nichols}  
\address{Mathematics Department, The University of Tennessee at Chattanooga, 
415 EMCS Building, Dept. 6956, 615 McCallie Ave, Chattanooga, TN 37403, USA}
\email{\mailto{Roger-Nichols@utc.edu}}
\urladdr{\url{http://www.utc.edu/faculty/roger-nichols/index.php}} 

\dedicatory{Dedicated with great pleasure to Lev Aronovich Sakhnovich on the occasion of 
his 80th birthday.}

\thanks{To appear in {\it Recent Advances in Schur Analysis and Stochastic Processes - A 
Collection of Papers Dedicated to Lev Sakhnovich}, D.\ Alpay and B.\ Kirstein (eds.), Operator Theory: 
Advances and Applications, Birkh\"auser, Basel.}

\date{\today}
\subjclass[2010]{Primary: 47B10, 47G10, Secondary: 34B27, 34L40.}
\keywords{Modified Fredholm determinants, semi-separable kernels, Jost functions.}

\begin{abstract} 
We study the analog of semi-separable integral kernels in $\cH$ of the type 
\begin{equation*}
K(x,x')=\begin{cases} F_1(x)G_1(x'), & a<x'< x< b, \\ 
F_2(x)G_2(x'), & a<x<x'<b,  
\end{cases}   
\end{equation*}
where $-\infty\leq a<b\leq \infty$, and for a.e.\ $x \in (a,b)$, $F_j (x) \in \cB_2(\cH_j,\cH)$ and 
$G_j(x) \in \cB_2(\cH,\cH_j)$ 
such that $F_j(\cdot)$ and $G_j(\cdot)$ are uniformly measurable, and 
\begin{equation*}  
\|F_j( \cdot)\|_{\cB_2(\cH_j,\cH)} \in L^2((a,b)), \; 
\|G_j (\cdot)\|_{\cB_2(\cH,\cH_j)} \in L^2((a,b)), \quad j=1,2,   
\end{equation*}
with $\cH$ and $\cH_j$, $j=1,2$, complex, separable Hilbert spaces. Assuming 
that $K(\cdot, \cdot)$ generates a Hilbert--Schmidt operator $\bsK$ in 
$L^2((a,b);\cH)$, we derive the analog of the Jost--Pais reduction theory that 
succeeds in proving that the modified Fredholm determinant 
${\det}_{2, L^2((a,b);\cH)}(\bsI - \alpha \bsK)$, $\alpha \in \bbC$, naturally reduces 
to appropriate Fredholm determinants in the Hilbert spaces $\cH$ (and 
$\cH \oplus \cH$). 

Some applications to Schr\"odinger operators with operator-valued potentials are provided. 
\end{abstract}

\maketitle


\section{Introduction}  \lb{s1}

Lev A.\ Sakhnovich's contributions to analysis in general are legendary, including, in particular, 
fundamental results in interpolation theory, spectral and inverse spectral theory, canonical systems, 
integrable systems and nonlinear evolution equations, integral equations, stochastic processes, 
applications to statistical physics, and the list goes on and on (see, e.g., \cite{SSR13}--\cite{Sa12}, 
and the literature cited therein). Since integral operators frequently play a role in his research interests, 
we hope our modest contribution to semi-separable operators in infinite dimensions will create some joy 
for him.

The principal topic in this paper concerns semi-separable integral operators and 
their associated Fredholm determinants. In a nutshell, suppose that $\cH$ 
and $\cH_j$, $j=1,2$, are complex, separable Hilbert spaces, that 
$-\infty\leq a<b\leq \infty$, and introduce the semi-separable integral kernel
in $\cH$, 
\begin{equation*}
K(x,x')=\begin{cases} F_1(x)G_1(x'), & a<x'< x< b, \\ 
F_2(x)G_2(x'), & a<x<x'<b,  
\end{cases}   
\end{equation*}
where for a.e.\ $x \in (a,b)$, $F_j (x) \in \cB_2(\cH_j,\cH)$ and $G_j(x) \in \cB_2(\cH,\cH_j)$ 
such that $F_j(\cdot)$ and $G_j(\cdot)$ are uniformly measurable (i.e., measurable with 
respect to the uniform operator topology), and 
\begin{equation*}  
\|F_j( \cdot)\|_{\cB_2(\cH_j,\cH)} \in L^2((a,b)), \; 
\|G_j (\cdot)\|_{\cB_2(\cH,\cH_j)} \in L^2((a,b)), \quad j=1,2.    
\end{equation*}
Assuming that $K(\cdot, \cdot)$ generates a Hilbert--Schmidt operator $\bsK$ in 
$L^2((a,b);\cH)$, we derive the analog of the Jost--Pais reduction theory that 
naturally reduces the modified Fredholm determinant 
${\det}_{2, L^2((a,b);\cH)}(\bsI - \alpha \bsK)$, $\alpha \in \bbC$, to appropriate 
Fredholm determinants in the Hilbert spaces $\cH$ (and $\cH \oplus \cH$) as 
described in detail in Theorem \ref{tA.12} and Corollary \ref{cA.13}. For instance, we will 
prove the following remarkable abstract version of the Jost--Pais-type reduction of modified 
Fredholm determinants \cite{JP51} (see also \cite{Ge86}, \cite{GM03}, \cite{Ne80}, \cite{Si00}), 
\begin{align}
&{\det}_{2,L^2((a,b);\cH)}(\bsI - \alpha \bsK)    \no\\[1mm] 
\begin{split} 
&\quad = {\det}_{\cH_1}\bigg(I_{\cH_1} - \alpha \int_a^bdx\, G_1(x)\widehat F_1(x,\alpha)\bigg)
\exp\bigg(\alpha \int_a^bdx\, \tr_{\cH}(F_1(x)G_1(x))\bigg)     \\[1mm] 
&\quad = {\det}_{\cH_2}\bigg(I_{\cH_2} - \alpha \int_a^bdx\, G_2(x)\widehat F_2(x,\alpha)\bigg)
\exp\bigg(\alpha \int_a^bdx\, \tr_{\cH}(F_2(x)G_2(x))\bigg),    \lb{1.1}    
 \end{split} 
\end{align}
in Theorem \ref{tA.12}, where $\widehat F_1(\cdot; \alpha )$ and $\widehat F_2(\cdot; \alpha )$
are defined via the Volterra integral equations
\begin{align}
\widehat F_1(x; \alpha )&=F_1(x)-\alpha \int_x^b dx'\, H(x,x')\widehat F_1(x'; \alpha ), \lb{1.35} \\ 
\widehat F_2(x; \alpha )&=F_2(x)+\alpha \int_a^x dx'\, H(x,x')\widehat F_2(x'; \alpha ). \lb{1.36}  
\end{align}  
The analog of \eqref{1.1} in the case where $\bsK$ is a trace class operator in $L^2((a,b);\cH)$
was recently derived in \cite{CGPST13} (cf.\ Corollary \ref{cA.13}).  

Section \ref{s2} focuses on our abstract results on semi-separable operators in infinite dimensions 
and represents the bulk of this paper. In particular, we will derive \eqref{1.1} and additional variants 
of it in Theorem \ref{tA.12}, the principal new result of this paper. Section \ref{s3} then presents some 
applications to Schr\"odinger operators with operator-valued potentials on $\bbR$ and $(0,\infty)$.

\section{Semiseparable Operators and Reduction Theory for Fredholm 
Determinants} \lb{s2}

In this section we describe one of the basic tools in this paper: a reduction 
theory for (modified) Fredholm determinants that permits one to reduce (modified) Fredholm 
determinants in the Hilbert space $L^2((a,b);\cH)$ to those in the Hilbert 
space $\cH$, as described in detail in Theorem \ref{tA.12} and in Corollary \ref{cA.13}. 
More precisely, we focus on a particular set of Hilbert--Schmidt 
operators $\bsK$ in $L^2((a,b);\cH)$ with $\cB(\cH)$-valued semi-separable 
integral kernels (with $\cH$ a complex, separable Hilbert space, generally of 
infinite dimension) and show how to naturally reduce the Fredholm 
determinant ${\det}_{2, L^2((a,b);\cH)}(\bsI - \alpha \bsK)$, $\alpha \in \bbC$, 
to appropriate Fredholm determinants in Hilbert spaces $\cH$ and 
$\cH \oplus \cH$ (in fact, we will describe a slightly more general framework 
below).   

In our treatment we closely follow the approaches presented in Gohberg,
Goldberg, and Kaashoek \cite[Ch.\ IX]{GGK90} and Gohberg, Goldberg, and
Krupnik \cite[Ch.\ XIII]{GGK00} (see also \cite{GK84}), and especially, in 
\cite{GM03}, where the particular case $\dim(\cH) < \infty$ was treated in 
detail. Our treatment of the case $\dim(\cH) = \infty$ in this section closely 
follows the one in \cite{CGPST13} in the case where $\bsK$ is a trace 
class operator in $L^2((a,b);\cH)$. 

Next, we briefly summarize some of the notation used in this section: $\cH$ and $\cK$ 
denote separable, complex Hilbert spaces, $(\cdot,\cdot)_{\cH}$ represents the scalar product in $\cH$ 
(linear in the second argument), and $I_{\cH}$ is the identity operator in $\cH$.

If $T$ is a linear operator mapping (a subspace of) a Hilbert space into 
another, then $\dom(T)$ and $\ker(T)$ denote the domain and kernel (i.e., 
null space) of $T$. The closure of a closable operator $S$ is denoted by $\ol S$. 
The spectrum, essential spectrum, and resolvent set of a closed linear operator in a Hilbert space 
will be denoted by $\sigma(\cdot)$, $\sigma_{ess}(\cdot)$, and $\rho(\cdot)$, respectively.

The Banach spaces of bounded and compact linear operators between complex, separable Hilbert spaces $\cH$ and $\cK$ are denoted by $\cB(\cH,\cK)$ and $\cB_\infty(\cH,\cK)$, 
respectively, and the corresponding $\ell^p$-based trace ideals will be denoted by 
$\cB_p(\cH,\cK)$, $p>0$.  When $\cH=\cK$, we simply write $\cB(\cH)$, $\cB_{\infty}(\cH)$ 
and $\cB_p(\cH)$, $p>0$, respectively.  The spectral radius of $T\in \cB(\cH,\cK)$ is denoted 
by $\spr(T)$.  Moreover, ${\det}_{\cH}(I_\cH-A)$, and $\tr_{\cH}(A)$ denote the standard 
Fredholm determinant and the corresponding trace 
of a trace class operator $A\in\cB_1(\cH)$. Modified Fredholm determinants are denoted by 
${\det}_{k, \cH}(I_\cH-A)$, $A\in\cB_k(\cH)$, $k \in \bbN$, $k \geq 2$.

For reasons of brevity, for operator-valued functions that are measurable with 
respect to the uniform operator topology, we typically use the short cut uniformly measurable.  

Before setting up the basic formalism for this section, we state the following 
elementary result: 

\begin{lemma} \lb{lA.1} 
Let $\cH$ and $\cH'$ be complex, separable Hilbert spaces and and $-\infty\leq a<b\leq \infty$. 
Suppose that for a.e.\ $x \in (a,b)$, $F (x) \in \cB(\cH',\cH)$ and $G(x) \in \cB(\cH,\cH')$ 
with $F(\cdot)$ and $G(\cdot)$ uniformly measurable, and  
\begin{equation}  
\|F( \cdot)\|_{\cB(\cH',\cH)} \in L^2((a,b)), \; 
\|G (\cdot)\|_{\cB(\cH,\cH')} \in L^2((a,b)). \lb{A.5}
\end{equation}
Consider the integral operator $\bsS$ in $L^2((a,b);\cH)$ with $\cB(\cH)$-valued 
separable integral kernel of the type 
\begin{equation}
S(x,x') = F(x) G(x') \, \text{ for a.e.\ $x, x' \in (a,b)$.}      \lb{A.6a} 
\end{equation}
Then 
\begin{equation}
\bsS \in \cB\big(L^2((a,b);\cH)\big).   \lb{A.7a}
\end{equation}
\end{lemma}
\begin{proof}
Let $f\in L^2((a,b);\cH)$, then for a.e.\ $x \in (a,b)$, and any integral 
operator $\bsT$ in $L^2((a,b);\cH)$ with $\cB(\cH)$-valued integral kernel $T(\cdot \, , \cdot)$, 
one obtains  
\begin{align}
\|(\bsT f)(x)\|_{\cH} & \leq \int_a^b dx' \, \|T(x,x')\|_{\cB(\cH)}\|f(x')\|_{\cH}   \no \\
& \leq \bigg(\int_a^b dx' \, \|T(x,x')\|_{\cB(\cH)}^2\bigg)^{1/2} 
\bigg(\int_a^b dx'' \, \|f(x'')\|_{\cH}^2\bigg)^{1/2},      \lb{A.8a}
\end{align}
and hence,
\begin{equation}
\int_a^b dx \, \|(\bsT f)(x)\|_{\cH}^2 \leq \bigg[\int_a^b dx \int_a^b dx' \, \|T(x,x')\|_{\cB(\cH)}^2\bigg]   
\int_a^b dx'' \, \|f(x'')\|_{\cH}^2,       \lb{A.9aa} 
\end{equation}
yields $\bsT \in \cB(L^2((a,b);\cH))$ whenever 
$\Big[\int_a^b dx \int_a^b dx' \, \|T(x,x')\|_{\cB(\cH)}^2\Big] < \infty$, implying
\begin{equation}
\|\bsT\|_{\cB(L^2((a,b);\cH))} \leq 
\bigg(\int_a^b dx \int_a^b dx' \, \|T(x,x')\|_{\cB(\cH)}^2\bigg)^{1/2}.    \lb{A.10a} 
\end{equation}
Thus, using the special form \eqref{A.6a} of $\bsS$ implies 
\begin{align}
\|\bsS\|_{\cB(L^2((a,b);\cH))}^2 &\leq \int_a^b dx \int_a^b dx' \, \|S(x,x')\|_{\cB(\cH)}^2   \no \\
& = \int_a^b dx \int_a^b dx' \, \|F(x) G(x')\|_{\cB(\cH)}^2    \no \\
& \leq \int_a^b dx \, \|F(x)\|_{\cB(\cH',\cH)}^2 \int_a^b dx' \, \|G(x')\|_{\cB(\cH,\cH')}^2 < \infty. 
\lb{A.10A} 
\end{align}
\end{proof}

At this point we now make the following initial set of assumptions: 
 
\begin{hypothesis} \lb{hA.2}  
Let $\cH$ and $\cH_j$, $j=1,2$, be complex, separable Hilbert spaces and $-\infty\leq a<b\leq \infty$. 
Suppose that for a.e.\ $x \in (a,b)$, $F_j (x) \in \cB(\cH_j,\cH)$ and $G_j(x) \in \cB(\cH,\cH_j)$ 
such that $F_j(\cdot)$ and $G_j(\cdot)$ are uniformly measurable, and 
\begin{equation}  
\|F_j( \cdot)\|_{\cB(\cH_j,\cH)} \in L^2((a,b)), \; 
\|G_j (\cdot)\|_{\cB(\cH,\cH_j)} \in L^2((a,b)), \quad j=1,2. \lb{A.1}
\end{equation}
\end{hypothesis}

Given Hypothesis \ref{hA.2}, we introduce in $L^2((a,b);\cH)$ the operator 
\begin{equation}
(\bsK f)(x)=\int_a^b dx'\, K(x,x')f(x') \, \text{ for a.e.\ $x \in (a,b)$, $f\in L^2((a,b);\cH)$,} 
\end{equation}
with $\cB(\cH)$-valued semi-separable integral kernel
$K(\cdot,\cdot)$ defined by
\begin{equation}
K(x,x')=\begin{cases} F_1(x)G_1(x'), & a<x'< x< b, \\ 
F_2(x)G_2(x'), & a<x<x'<b. \end{cases}  \lb{A.3}
\end{equation}

The operator $\bsK$ is bounded,  
\begin{equation}
\bsK \in \cB\big(L^2((a,b);\cH)\big). 
\end{equation}
In fact, using \eqref{A.10a} and \eqref{A.3}, one readily verifies
\begin{align}
\begin{split} 
& \int_a^bdx\int_a^bdx'\|K(x,x')\|_{\cB(\cH)}^2 
= \int_a^b dx \bigg(\int_a^x + \int_x^b \bigg)dx'\|K(x,x')\|_{\cB(\cH)}^2\\
& \quad \leq \sum_{j=1}^2 \int_a^bdx \|F_j(x)\|_{\cB(\cH_j,\cH)}^2 
\int_a^b dx'\|G_j(x')\|_{\cB(\cH,\cH_j)}^2 < \infty. 
\end{split} 
\end{align}

Associated with $\bsK$ we also introduce the bounded Volterra operators $\bsH_a$ and
$\bsH_b$ in $L^2((a,b);\cH)$ defined by 
\begin{align}
(\bsH_af)(x)&=\int_a^x dx'\, H(x,x')f(x'), \lb{A.4} \\
(\bsH_bf)(x)&=-\int_x^b dx'\, H(x,x')f(x'); \quad f\in
L^2((a,b);\cH), \lb{A.12a} 
\end{align}
with $\cB(\cH)$-valued (triangular) integral kernel
\begin{equation}
H(x,x')=F_1(x)G_1(x')-F_2(x)G_2(x').  \lb{A.6}
\end{equation} 
Moreover, introducing the bounded operator block matrices\footnote{$M^\top$  
denotes the transpose of the operator matrix $M$.}  
\begin{align}
C(x)&=(F_1(x) \;\; F_2(x)), \lb{A.7} \\
B(x)&=(G_1(x) \;\; -G_2(x))^\top, \lb{A.8} 
\end{align}
one verifies
\begin{equation}
H(x,x')=C(x)B(x'), \, \text{ where } \begin{cases} 
a<x'<x<b & \text{for $\bsH_a$,} \\ 
a<x<x'<b & \text {for $\bsH_b$} \end{cases} \lb{A.9}
\end{equation}
and 
\begin{equation}
K(x,x')=\begin{cases} C(x)(I_{\cH_1 \oplus \cH_2}-P_0)B(x'), & a<x'<x<b, \\
-C(x)P_0B(x'), & a<x<x'<b, \end{cases} \lb{A.9a}
\end{equation}
with 
\begin{equation}
 P_0=\begin{pmatrix} 0 & 0 \\ 0 & I_{\cH_2} \end{pmatrix}. \lb{A.9b} 
\end{equation}

The next result proves that, as expected, $\bsH_a$ and $\bsH_b$ are quasi-nilpotent (i.e., have 
vanishing spectral radius) in $L^2((a,b);\cH)$:

\begin{lemma} \lb{lA.2a}
Assume Hypothesis \ref{hA.2}. Then $\bsH_a$ and $\bsH_b$ are quasi-nilpotent in 
$L^2((a,b);\cH)$, equivalently,
\begin{equation}
\sigma (\bsH_a) = \sigma (\bsH_b) = \{0\}.     \lb{A.20a} 
\end{equation}
\end{lemma}
\begin{proof}
It suffices to discuss $\bsH_a$. Then estimating the norm of $H_a^n (x,x')$, $n \in \bbN$, (i.e., the integral kernel for $\bsH_a^n$) in a straightforward 
manner (cf.\ \eqref{A.4}, \eqref{A.6}) yields for a.e.\ $x,x' \in (a,b)$, 
\begin{align}
& \big\|H_a^n (x,x')\big\|_{\cB(\cH)} \leq 2^n 
\max_{j = 1,2}\big(\|F_j(x)\|_{\cB(\cH_j,\cH)}\big) \max_{k = 1,2}\big(\|G_k(x')\|_{\cB(\cH,\cH_k)}\big) 
\no \\
& \quad \times \f{1}{(n-1)!} \bigg[\int_a^x dx'' \, 
\max_{1 \leq \ell, m \leq 2}\big(\|G_{\ell}(x'')\|_{\cB(\cH,\cH_{\ell})} 
\|F_m (x'')\|_{\cB(\cH_m,\cH)}\big)\bigg]^{(n-1)},     \no \\
& \hspace*{10cm}  n \in \bbN.    
\end{align}
Thus, applying \eqref{A.10a}, one verifies
\begin{align}
& \big\|\bsH_a^n\big\|_{\cB(L^2((a,b);\cH))} \leq 
\bigg(\int_a^b dx \int_a^b dx' \, \|H_a^n(x,x')\|_{\cB(\cH)}^2\bigg)^{1/2}     \no \\
& \quad \leq \max_{j=1,2} \bigg(\int_a^b dx \, \|F_j(x)\|_{\cB(\cH_j,\cH)}\bigg)^{1/2} 
\max_{k=1,2} \bigg(\int_a^b dx' \, \|G_k(x')\|_{\cB(\cH,\cH_k)}\bigg)^{1/2}    \no \\
& \qquad \times \f{2^n}{(n-1)!} \max_{1 \leq \ell,m \leq 2} \bigg(\int_a^b dx'' 
\|G_{\ell}(x'')\|_{\cB(\cH,\cH_{\ell})} \|F_m(x'')\|_{\cB(\cH_m,\cH)}\bigg)^{(n-1)},     \no \\
& \hspace*{10cm} n \in\bbN,    
\end{align}
and hence
\begin{equation}
\spr(\bsH_a) = \lim_{n\to\infty} \big\|\bsH_a^n\big\|_{\cB(L^2((a,b);\cH))}^{1/n} = 0  
\end{equation}
(where $\spr (\, \cdot \,)$ abbreviates the spectral radius). 
Thus, $\bsH_a$ and $\bsH_b$ are quasi-nilpotent in $L^2((a,b);\cH)$ which in turn is equivalent to 
\eqref{A.20a}.
\end{proof}

Next, introducing the linear maps
\begin{align}
&Q\colon \cH_2\mapsto L^2((a,b);\cH), \quad (Q w)(x)=F_2(x) w,
\quad w \in\cH_2, \lb{A.10} \\
&R\colon L^2((a,b);\cH) \mapsto \cH_2, \quad (Rf)=\int_a^b
dx'\,G_2(x')f(x'), \quad f\in L^2((a,b);\cH), \lb{A.11} \\
&S\colon \cH_1 \mapsto L^2((a,b);\cH), \quad (S v)(x)=F_1(x) v,
\quad v\in\cH_1, \lb{A.12} \\
&T\colon L^2((a,b);\cH) \mapsto \cH_1, \quad (Tf)=\int_a^b 
dx'\, G_1(x')f(x'), \quad f\in L^2((a,b);\cH), \lb{A.13} 
\end{align}
one easily verifies the following elementary result (cf.\ 
\cite[Sect.\ IX.2]{GGK90}, \cite[Sect.\ XIII.6]{GGK00} in the case $\dim(\cH)<\infty$):

\begin{lemma} \lb{lA.3} Assume Hypothesis \ref{hA.2}. Then
\begin{align}
\bsK &=\bsH_a + QR \lb{A.14} \\
 &=\bsH_b + ST. \lb{A.15}
\end{align}
\end{lemma}

To describe the inverse of $\bsI-\alpha \bsH_a$ and $\bsI-\alpha \bsH_b$,
$\alpha\in\bbC$, one introduces the block operator matrix $A(\cdot)$ in $\cH_1 \oplus \cH_2$
\begin{align}
A(x)&=\begin{pmatrix} G_1(x)F_1(x) & G_1(x)F_2(x) \\ 
-G_2(x)F_1(x) & -G_2(x)F_2(x) \end{pmatrix} \lb{A.16} \\[1mm]
& = B(x)C(x)\, \text{ for a.e.\ $x\in (a,b)$} \lb{A.17}
\end{align}
and considers the linear evolution equation in $\cH_1 \oplus \cH_2$, 
\begin{equation}
\begin{cases} u'(x) = \alpha A(x) u(x), \quad \alpha \in \bbC, \, \text{ for a.e.\ $x \in (a,b)$,}  \\
u(x_0) = u_0 \in \cH_1 \oplus \cH_2  \end{cases}    \lb{A.32a} 
\end{equation}
for some $x_0 \in (a,b)$. Since $A(x) \in \cB(\cH_1 \oplus \cH_2)$ for a.e. $x\in (a,b)$, 
$A(\cdot)$ is uniformly measurable, and $\|A(\cdot)\|_{\cB(\cH_1 \oplus \cH_2)} \in L^1((a,b))$, 
Theorems\ 1.1 and 1.4 in \cite{OY11} (see also \cite{HO00}, which includes a discussion of a 
nonlinear extension of \eqref{A.32a}) apply and yield the existence of a unique propagator 
$U(\,\cdot \, ,\,\cdot \, ; \alpha)$ on $(a,b) \times (a,b)$ satisfying the following conditions: 
\begin{align}
& U(\cdot \, ,\cdot \, ; \alpha):(a,b) \times (a,b) \to \cB(\cH_1 \oplus \cH_2) \, \text{ is uniformly (i.e., norm) continuous.}  
\lb{A.33A} \\
& \text{There exists $C_{\alpha} > 0$ such that for all $x, x' \in (a,b)$, } \, 
\|U(x,x'; \alpha)\|_{\cB(\cH)} \leq C_{\alpha}.    \lb{A.33B} \\
& \text{For all $x, x', x'' \in (a,b)$, } \, U(x,x'; \alpha) U(x',x''; \alpha) = U(x,x''; \alpha),   \lb{A.33C} \\ 
& \hspace*{6.21cm} U(x,x; \alpha) = I_{\cH_1 \oplus \cH_2}.    \no \\ 
&\text{For all $u \in \cH_1 \oplus \cH_2$, $\alpha \in \bbC$, }\no\\
&\quad U(x,\cdot\, ;\alpha) u, U(\cdot \, ,x;\alpha) u \in W^{1,1} ((a,b); \cH_1 \oplus \cH_2),\quad x\in (a,b), 
\lb{A.33D} \\ 
& \text{and}    \no \\
& \text{for a.e.\ $x \in (a,b)$, } \, (\partial/\partial x) U(x,x'; \alpha) u = \alpha A(x) U(x,x'; \alpha) u, 
\quad x' \in (a,b),   \lb{A.33E} \\
& \text{for a.e.\ $x' \in (a,b)$, } \, (\partial/\partial x') U(x,x'; \alpha) u = - \alpha U(x,x'; \alpha) A(x') u, 
\quad x \in (a,b).   \lb{A.33F}
\end{align} 

Hence, $u(\, \cdot \, ; \alpha)$ defined by
\begin{equation}
u(x; \alpha) = U(x,x_0; \alpha) u_0, \quad x \in (a,b),      \lb{A.34a} 
\end{equation}
is the unique solution of \eqref{A.32a}, satisfying
\begin{equation}
u(\,\cdot\,; \alpha) \in W^{1,1} ((a,b); \cH_1 \oplus \cH_2).     \lb{A.35a} 
\end{equation}

In fact, an explicit construction (including the proof of uniqueness and that of the properties of 
\eqref{A.33A}--\eqref{A.33F}) of $U(\cdot\, , \cdot \, ; \alpha)$ can simply be obtained by a 
norm-convergent iteration of 
\begin{equation}
U(x,x'; \alpha) = I_{\cH_1 \oplus \cH_2} + \alpha \int_{x'}^x dx'' \, A(x'') U(x'', x'; \alpha), 
\quad x, x' \in (a,b).      \lb{A.36a} 
\end{equation}
Moreover, because of the integrability assumptions made in Hypothesis \ref{hA.2}, 
\eqref{A.32a}-\eqref{A.36a} extend to $x, x' \in [a,b)$ (resp., $x, x' \in (a,b]$) if $a > - \infty$ 
(resp., $b < \infty$) and permit taking norm limits of $U(x,x'; \alpha)$ as $x, x'$ to $-\infty$ if 
$a=-\infty$ (resp., $+\infty$ if $b= +\infty$), see also Remark \ref{rA.5}. 

The next result appeared in \cite[Sect.\ IX.2]{GGK90}, 
\cite[Sects.\ XIII.5, XIII.6]{GGK00} in the special case $\dim(\cH)<\infty$:

\begin{theorem} \lb{tA.4} 
Assume Hypothesis \ref{hA.2}. Then, \\
$(i)$ $\bsI-\alpha \bsH_a$ and $\bsI-\alpha \bsH_b$ are boundedly invertible for all
$\alpha\in\bbC$ and 
\begin{align}
(\bsI-\alpha \bsH_a)^{-1}&= \bsI+\alpha \bsJ_a(\alpha), \lb{A.19} \\
(\bsI-\alpha \bsH_b)^{-1}&= \bsI+\alpha \bsJ_b(\alpha), \lb{A.20} \\
(\bsJ_a(\alpha) f)(x)&=\int_a^x dx'\, J(x,x'; \alpha )f(x'), \lb{A.21} \\ 
(\bsJ_b(\alpha) f)(x)&=-\int_x^b dx'\, J(x,x'; \alpha )f(x'); \quad f \in
L^2((a,b);\cH), \lb{A.22} \\  
J(x,x'; \alpha )&=C(x) U(x, x'; \alpha) B(x'), \, \text{ where }
\begin{cases}  a<x'<x<b & \text{for $\bsJ_a(\alpha)$,} \\ 
a<x<x'<b & \text {for $\bsJ_b(\alpha)$.} \end{cases}  \lb{A.23}  
\end{align}
$(ii)$ Let $\alpha\in\bbC$. Then $\bsI-\alpha \bsK$ is boundedly invertible if and only if 
$I_{\cH_2}-\alpha R(\bsI-\alpha \bsH_a)^{-1}Q$ is. Similarly,
$\bsI-\alpha \bsK$ is boundedly invertible if and only if 
$I_{\cH_1}-\alpha T(\bsI-\alpha \bsH_b)^{-1}S$ is. 
In particular,
\begin{align}
& (\bsI-\alpha \bsK)^{-1}=(\bsI-\alpha \bsH_a)^{-1}
+\alpha (\bsI-\alpha \bsH_a)^{-1}QR(\bsI-\alpha \bsK)^{-1} \lb{A.24} \\
& \quad =(\bsI-\alpha \bsH_a)^{-1} \no \\
& \qquad +\alpha (\bsI-\alpha \bsH_a)^{-1}Q\big[I_{\cH_2}
-\alpha R(\bsI-\alpha \bsH_a)^{-1}Q\big]^{-1}R(\bsI-\alpha \bsH_a)^{-1}
\lb{A.25} \\
& \quad =(\bsI-\alpha \bsH_b)^{-1}+\alpha (\bsI-\alpha \bsH_b)^{-1}
ST(\bsI-\alpha \bsK)^{-1} \lb{A.26} \\
& \quad =(\bsI-\alpha \bsH_b)^{-1} \no \\
& \qquad +\alpha (\bsI-\alpha \bsH_b)^{-1}S
\big[I_{\cH_1}-\alpha T(\bsI-\alpha \bsH_b)^{-1}S\big]^{-1}
T(\bsI-\alpha \bsH_b)^{-1}.
\lb{A.27}
\end{align}
\end{theorem}
\begin{proof} 
To prove the results \eqref{A.19}--\eqref{A.23} it suffices to focus on $\bsH_a$. Let 
$f \in L^2((a,b); \cH)$. Then using $H(x,x') = C(x) B(x')$ and 
$A(x) = B(x) C(x)$  (cf.\ \eqref{A.9} 
and \eqref{A.17}) one computes (for some $x_0 \in (a,b)$) with the help of \eqref{A.33E}, 
\begin{align}
& \big((\bsI - \alpha \bsH_a)(\bsI + \alpha \bsJ_a(\alpha)) f\big)(x) 
= f(x) - \alpha \int_a^x dx' \, C(x) B(x') f(x') \no \\
& \qquad + \alpha \int_a^x dx' \, C(x) U(x,x'; \alpha) B(x') f(x')    \no \\ 
& \qquad - \alpha^2 \int_a^x dx' \, C(x) B(x') \int_a^{x'} dx'' \, C(x') U(x',x''; \alpha) B(x'') f(x'')  \no \\
& \quad = f(x) - \alpha \int_a^x dx' \, C(x) B(x') f(x') 
+ \alpha \int_a^x dx' \, C(x) U(x,x'; \alpha) B(x') f(x')    \no \\ 
& \qquad - \alpha^2 \int_a^x dx' \, C(x)B(x')C(x')U(x',x_0;\alpha) 
\int_a^{x'} dx'' \, U(x_0,x''; \alpha) B(x'') f(x'')   \no \\
& \quad = f(x) - \alpha \int_a^x dx' \, C(x) B(x') f(x') 
+ \alpha \int_a^x dx' \, C(x) U(x,x'; \alpha) B(x') f(x')    \no \\ 
& \qquad - \alpha \int_a^x dx' \, C(x) [(\partial/\partial x') U(x',x_0; \alpha)]
\int_a^{x'} dx'' \, U(x_0,x''; \alpha) B(x'') f(x'')   \no \\
& \quad = f(x) - \alpha \int_a^x dx' \, C(x) B(x') f(x') 
+ \alpha \int_a^x dx' \, C(x) U(x,x'; \alpha) B(x') f(x')    \no \\ 
& \qquad - \alpha C(x) \bigg[U(x',x_0; \alpha) \int_a^{x'} dx'' \, U(x_0,x''; \alpha) B(x'') f(x'') \bigg|_{x'=a}^x 
\no \\[1mm] 
& \hspace*{2.3cm} - \int_a^x dx' \, U(x',x_0; \alpha) U(x_0,x'; \alpha) B(x') f(x')\bigg]    \no \\
& \quad = f(x) \, \text{ for a.e.\ $x \in (a,b)$.} 
\end{align}
In the same manner one proves 
\begin{equation}
\big((\bsI + \alpha \bsJ_a(\alpha))(\bsI - \alpha \bsH_a) f\big)(x) = f(x)  \, \text{ for a.e.\ $x \in (a,b)$.} 
\end{equation} 

By \eqref{A.14} and \eqref{A.15}, $\bsK - \bsH_a$ and $\bsK - \bsH_b$ factor into $QR$ and $ST$, respectively.  Consequently, \eqref{A.24} and \eqref{A.26} follow from the second resolvent identity, while \eqref{A.25} and \eqref{A.27} are direct applications of Kato's resolvent equation for factored 
perturbations (cf. \cite[Sect.~2]{GLMZ05}). 
\end{proof}

\begin{remark} \lb{rA.5}  
Even though this will not be used in this paper, we mention for completeness that if 
$(\bsI - \alpha \bsK)^{-1} \in \cB\big(L^2((a,b);\cH)\big)$, and if $U(\cdot\, ,a ; \alpha)$ 
is defined by 
\begin{equation}
U(x,a; \alpha) = I_{\cH_1 \oplus \cH_2} + \alpha \int_a^x dx' \, A(x') U(x',a; \alpha), \quad x \in (a,b), 
\lb{A.48a}
\end{equation}
and partitioned with respect to $\cH_1 \oplus \cH_2$ as
\begin{equation}
U(x,a; \alpha) = \begin{pmatrix} U_{1,1}(x,a; \alpha) & U_{1,2}(x,a;\alpha) \\
U_{2,1}(x,a; \alpha) & U_{2,2}(x,a;\alpha) \end{pmatrix}, \quad x \in (a,b), 
\end{equation}
then 
\begin{align}
(\bsI - \alpha \bsK)^{-1} &= \bsI+\alpha \bsL(\alpha), \lb{A.28} \\
(\bsL(\alpha) f)(x) &= \int_a^b dx'\, L(x,x'; \alpha )f(x'), \lb{A.29} \\  
L(x,x'; \alpha )  &= \begin{cases} C(x)U(x,a; \alpha) (I-P(\alpha))
U(x',a; \alpha)^{-1}B(x'), & a<x'<x<b, \\  
-C(x)U(x,a; \alpha) P(\alpha) U(x',a; \alpha)^{-1}B(x'), & a<x<x'<b, \end{cases} 
\lb{A.30}  
\end{align}
where 
\begin{equation}
P(\alpha)=\begin{pmatrix} 0 & 0 \\ U_{2,2}(b,a; \alpha )^{-1}
U_{2,1}(b,a; \alpha ) & I_{\cH_2} \end{pmatrix},     \lb{A.33}
\end{equation}
with $U(b,a; \alpha) = \nlim_{x \uparrow b} U(x,a; \alpha)$. (Here $\nlim$ abbreviates the limit 
in the norm topology.)
These results can be shown as in the finite-dimensional case treated in \cite[Ch.\ IX]{GGK90}.
\end{remark}

\begin{lemma} \lb{lA.6}
Assume Hypothesis \ref{hA.2} and introduce, for $\alpha\in\bbC$ and a.e.\ $x \in (a,b)$, the Volterra 
integral equations
\begin{align}
\widehat F_1(x; \alpha )&=F_1(x)-\alpha \int_x^b dx'\, H(x,x')\widehat F_1(x'; \alpha ), \lb{A.35} \\ 
\widehat F_2(x; \alpha )&=F_2(x)+\alpha \int_a^x dx'\, H(x,x')\widehat F_2(x'; \alpha ). \lb{A.36}  
\end{align}  
Then there exist unique a.e.\ solutions on $(a,b)$, $\widehat F_j (\cdot\, ; \alpha) \in \cB(\cH_j,\cH)$,  
of \eqref{A.35}, \eqref{A.36} such that $\widehat F_j(\cdot\, ; \alpha)$ are uniformly measurable, and 
\begin{equation}  
\big\|\widehat F_j( \cdot \, ; \alpha)\big\|_{\cB(\cH_j,\cH)} \in L^2((a,b)), \quad j=1,2. \lb{A.45}
\end{equation}
\end{lemma}
\begin{proof}
Introducing, 
\begin{align}
\widehat F_{1,0} (x; \alpha) &= F_1(x),    \no \\
\widehat F_{1,n} (x; \alpha) &= - \alpha \int_x^b dx' \, H(x,x') \widehat F_{1,n-1} (x'; \alpha), \quad n \in\bbN,    \\
\widehat F_{2,0} (x; \alpha) &= F_2(x),    \no \\
\widehat F_{2,n} (x; \alpha) &= \alpha \int_a^x dx' \, H(x,x') \widehat F_{2,n-1} (x'; \alpha), \quad n \in\bbN,
\end{align}
for a.e.\ $x \in (a,b)$, the familiar iteration procedure (in the scalar or matrix-valued context) yields 
for fixed $x \in (a,b)$ except for a set of Lebesgue measure zero, 
\begin{align}
& \big\|\widehat F_{1,n} (x; \alpha)\big\|_{\cB(\cH_1,\cH)} \leq (2 |\alpha|)^n 
\max_{j=1,2} \big(\|F_j(x)\|_{\cB(\cH_j,\cH)}\big)     \\ 
& \quad \times \f{1}{n!} 
\bigg[\int_x^b dx' \, \max_{1 \leq k, \ell \leq 2} \big(\|G_k (x')\|_{\cB(\cH,\cH_k)} 
\|F_{\ell} (x')\|_{\cB(\cH_{\ell},\cH)}\big)\bigg]^n, \quad n \in \bbN,     \no \\
& \big\|\widehat F_{2,n} (x; \alpha)\big\|_{\cB(\cH_2,\cH)} \leq (2 |\alpha|)^n 
\max_{j=1,2} \big(\|F_j (x)\|_{\cB(\cH_j,\cH)}\big)      \\ 
& \quad \times \f{1}{n!} 
\bigg[\int_a^x dx' \, \max_{1 \leq k, \ell \leq 2} \big(\|G_k (x')\|_{\cB(\cH,\cH_k)} 
\|F_{\ell} (x')\|_{\cB(\cH_{\ell},\cH)}\big)\bigg]^n, \quad n \in \bbN.    \no 
\end{align}
Thus, the norm convergent expansions 
\begin{align}
\widehat F_j (x; \alpha) = \sum_{n=0}^{\infty} \widehat F_{j,n}(x; \alpha), \quad j=1,2, \, 
\text{ for a.e.\ $x \in (a,b)$,}  
\lb{A.50}
\end{align}
yield the bounds
\begin{align}
\big\|\widehat F_j (x; \alpha)\big\|_{\cB(\cH_j,\cH)} &\leq \max_{k=1,2} \big(\|F_k(x)\|_{\cB(\cH_k,\cH)}\big)   \\
& \quad \times \max_{1 \leq \ell, m \leq 2}
\exp\bigg(2 |\alpha| \int_a^b dx' \|G_{\ell} (x')\|_{\cB(\cH,\cH_{\ell})} 
\|F_m (x')\|_{\cB(\cH_m,\cH)}\bigg)      \no 
\end{align}
for a.e.\ $x \in (a,b)$. As in the scalar case (resp., as in the proof of Theorem \ref{tA.4})  
one shows that \eqref{A.50} uniquely satisfies \eqref{A.35}, \eqref{A.36} 
\end{proof}

\begin{lemma} \lb{lA.7}
Assume Hypothesis \ref{hA.2}, let $\alpha\in\bbC$, and introduce 
\begin{align}
& U(x; \alpha)=\begin{pmatrix} I_{\cH_1}-\alpha \int_x^b dx'\, G_1(x') 
\widehat F_1(x'; \alpha ) & \alpha\int_a^x dx'\, G_1(x')\widehat F_2(x'; \alpha ) \\ 
\alpha\int_x^b dx'\, G_2(x')\widehat F_1(x'; \alpha ) & I_{\cH_2}-\alpha \int_a^x
dx'\, G_2(x')\widehat F_2(x'; \alpha ) \end{pmatrix},  \no \\
& \hspace*{9.4cm} x\in (a,b). \lb{A.37}  
\end{align}
If 
\begin{align}
\bigg[I_{\cH_1}-\alpha \int_a^b dx\, G_1(x) \widehat F_1(x; \alpha )\bigg]^{-1} 
\in \cB(\cH_1),    \lb{A.38} \\
\intertext{or equivalently,} 
\bigg[I_{\cH_2}-\alpha \int_a^b dx\, G_2(x)\widehat F_2(x; \alpha )\bigg]^{-1} \in \cB(\cH_2), \lb{A.39} 
\end{align}
then 
\begin{equation}
U(a; \alpha) , \, U(b; \alpha), \, U(x; \alpha), \; x\in (a,b), 
\end{equation} 
are boundedly invertible in $\cH_1 \oplus \cH_2$. In particular,  
\begin{equation}
U(x,x'; \alpha) = U(x; \alpha) U(x'; \alpha)^{-1}, \quad x, x' \in (a,b),   \lb{A.72a}
\end{equation}   
is the propagator for the evolution equation \eqref{A.32a} satisfying \eqref{A.33A}--\eqref{A.36a}, and 
\eqref{A.72a} extends by norm continuity to $x,x' \in \{a,b\}$. 
\end{lemma}
\begin{proof}
Since 
\begin{equation}
U(a; \alpha)=\begin{pmatrix} I_{\cH_1}-\alpha \int_a^b dx'\, G_1(x') 
\widehat F_1(x'; \alpha ) & 0 \\ 
\alpha\int_a^b dx'\, G_2(x')\widehat F_1(x'; \alpha ) & I_{\cH_2} \end{pmatrix},     \lb{A.68a}  
\end{equation}
the operator $U(a; \alpha)$ is boundedly invertible in $\cH_1 \oplus \cH_2$ if and only if 
$\Big[I_{\cH_1}-\alpha \int_a^b dx'\, G_1(x') \widehat F_1(x'; \alpha )\Big]$ is boundedly invertible 
in $\cH_1$. (One recalls that 
a bounded $2 \times 2$ block operator 
$D = \left(\begin{smallmatrix} D_{1,1} & 0 \\ D_{2,1} & I_{\cH_2} \end{smallmatrix}\right)$ 
in $\cH_1 \oplus \cH_2$ is boundedly invertible if and only if $D_{1,1}$ is boundedly invertible in $\cH_1$, with 
$D^{-1} = \left(\begin{smallmatrix} D_{1,1}^{-1} & 0 \\ - D_{2,1} D_{1,1}^{-1}  & I_{\cH_2} 
\end{smallmatrix}\right)$ if $D$ is boundedly invertible.) Similarly, 
\begin{equation}
U(b; \alpha)=\begin{pmatrix} I_{\cH_1} & \alpha\int_a^b dx'\, G_1(x')\widehat F_2(x'; \alpha ) \\ 
0 & I_{\cH_2}-\alpha \int_a^b dx'\, G_2(x')\widehat F_2(x'; \alpha ) \end{pmatrix}     \lb{A.69a}  
\end{equation}
is boundedly invertible in $\cH_1 \oplus \cH_2$ if and only if 
$\Big[I_{\cH_2}-\alpha \int_a^b dx'\, G_2(x') \widehat F_2(x'; \alpha )\Big]$ is in $\cH_2$. (Again, one 
recalls that a bounded $2 \times 2$ block operator 
$E = \left(\begin{smallmatrix} I_{\cH_1} & E_{1,2} \\ 0 & E_{2,2} \end{smallmatrix}\right)$ 
in $\cH_1 \oplus \cH_2$ is boundedly invertible if and only if $E_{2,2}$ is boundedly invertible in $\cH_2$, with 
$E^{-1} = \left(\begin{smallmatrix} I_{\cH_1} & - E_{1,2} E_{2,2}^{-1} \\ 0 & E_{2,2}^{-1} 
\end{smallmatrix}\right)$ if $E$ is boundedly invertible.)

The equivalence of \eqref{A.38} and \eqref{A.39} has been settled in Theorem\ \ref{tA.4}\,$(ii)$. 

Next, differentiating the entries on the right-hand side of \eqref{A.37} with respect to $x$ and 
using the Volterra integral equations \eqref{A.35}, \eqref{A.36} yields  
\begin{equation}
(d/dx) U(x; \alpha) u = \alpha A(x) U(x; \alpha) u \, \text{ for a.e.\ $x \in (a,b)$}.  
\end{equation}
Thus, by uniqueness of the propagator $U(\cdot \, , \cdot \, ; \alpha)$, extended by norm 
continuity to $x=a$ (cf.\ Remark \ref{rA.5}), one obtains that 
\begin{equation}
U(x,a; \alpha) = U(x; \alpha) U(a; \alpha)^{-1}, \quad x \in (a,b). 
\end{equation}
Thus, $U(x; \alpha) = U(x,a; \alpha) U(a; \alpha)$ is boundedly invertible for all $x \in (a,b)$ since 
$U(x,a; \alpha)$, $x \in (a,b)$ is by construction (using norm continuity and the transitivity property 
in \eqref{A.33C}), and $U(a; \alpha)$ is boundedly invertible by hypothesis.  
Consequently, once more by uniqueness of the propagator $U(\cdot \, , \cdot \, ; \alpha)$, one 
obtains that  
\begin{equation}
U(x,x'; \alpha) = U(x; \alpha) U(x'; \alpha)^{-1}, \quad x, x' \in (a,b).   \lb{A.77a}
\end{equation}
Again by norm continuity, \eqref{A.77a} extends to $x, x' \in \{a, b\}$. 
\end{proof}

In the special case where $\cH$ and $\cH_j$, $j=1,2$, are finite-dimensional, the Volterra integral 
equations \eqref{A.35}, \eqref{A.36} and the operator $U$ in \eqref{A.37} were introduced 
in \cite{GM03}. 

\begin{lemma} \lb{lA.9} 
Let $\cH$ and $\cH'$ be complex, separable Hilbert spaces and $-\infty\leq a<b\leq \infty$. 
Suppose that for a.e.\ $x \in (a,b)$, $F (x) \in \cB_2(\cH',\cH)$ and $G(x) \in \cB_2(\cH,\cH')$ 
with $F(\cdot)$ and $G(\cdot)$ weakly measurable, and  
\begin{equation}  
\|F( \cdot)\|_{\cB_2(\cH',\cH)} \in L^2((a,b)), \; 
\|G (\cdot)\|_{\cB_2(\cH,\cH')} \in L^2((a,b)). \lb{A.79a}
\end{equation}
Consider the integral operator $\bsS$ in $L^2((a,b);\cH)$ with $\cB_1(\cH)$-valued 
separable integral kernel of the type 
\begin{equation}
S(x,x') = F(x) G(x') \, \text{ for a.e.\ $x, x' \in (a,b)$.}      \lb{A.80a} 
\end{equation}
Then 
\begin{equation}
\bsS \in \cB_1\big(L^2((a,b);\cH)\big).   \lb{A.81a}
\end{equation}
\end{lemma}
\begin{proof}
Since the Hilbert space of Hilbert--Schmidt operators, $\cB_2(\cH',\cH)$, is separable, weak 
measurability of $F(\cdot)$ implies $\cB_2(\cH',\cH)$-measurability 
by Pettis' theorem (cf., e.g., \cite[Theorem\ 1.1.1]{ABHN01}, \cite[Theorem\ II.1.2]{DU77}, 
\cite[3.5.3]{HP85}), and analogously for $G(\cdot)$. 

Next, one introduces (in analogy to \eqref{A.10}--\eqref{A.13}) the linear operators
\begin{align}
& Q_F \colon \cH' \mapsto L^2((a,b);\cH), \quad (Q_F w)(x)=F(x) w,
\quad w \in\cH', \lb{A.81A} \\
& R_G \colon L^2((a,b);\cH) \mapsto \cH', \quad (R_G f)=\int_a^b
dx'\,G(x')f(x'), \quad f\in L^2((a,b);\cH), \lb{A.81B}
\end{align} 
such that
\begin{equation}
\bsS = Q_F R_G.    \lb{A.81C} 
\end{equation}
Thus, with $\{v_n\}_{n\in\bbN}$ a complete orthonormal system in $\cH'$, using the monotone 
convergence theorem, one concludes that 
\begin{align}
& \|Q_F\|_{\cB_2(\cH', L^2((a,b);\cH))}^2 = \sum_{n\in\bbN} \|Q_F v_n\|_{L^2((a,b);\cH)}^2  \no \\
& \quad = \sum_{n\in\bbN} \int_a^b dx \, \|F(x) v_n\|_{\cH}^2 
= \int_a^b dx \, \sum_{n\in\bbN} (v_n, F(x)^* F(x) v_n)_{\cH'}    \no \\
& \quad = \int_a^b dx \, {\tr}_{\cH'} \big(F(x)^* F(x)\big) = \int_a^b dx \, \big\|F(x)^* F(x)\big\|_{\cB_1(\cH')} 
\no \\ 
& \quad = \int_a^b dx \, \|F(x)\|_{\cB_2(\cH', \cH)}^2 < \infty.     \lb{A.83a} 
\end{align}
The same argument applied to $R_G^*$ (which is of the form $Q_{G^*}$, i.e., given by 
\eqref{A.81A} with $F(\cdot)$ replaced by $G(\cdot)^*$) then proves 
$R_G^* \in \cB_2(\cH', L^2((a,b);\cH))$. Hence,  
\begin{equation}
Q_F \in \cB_2(\cH', L^2((a,b);\cH)), \quad R_G \in \cB_2(L^2((a,b);\cH), \cH'),  
\end{equation}
together with the factorization \eqref{A.81C}, prove \eqref{A.81a}.
\end{proof}

Next, we strengthen our assumptions as follows: 

\begin{hypothesis} \lb{hA.10}  
Let $\cH$ and $\cH_j$, $j=1,2$, be complex, separable Hilbert spaces and $-\infty\leq a<b\leq \infty$. 
Suppose that for a.e.\ $x \in (a,b)$, $F_j (x) \in \cB_2(\cH_j,\cH)$ and $G_j(x) \in \cB_2(\cH,\cH_j)$ 
such that $F_j(\cdot)$ and $G_j(\cdot)$ are weakly measurable, and 
\begin{equation}  
\|F_j( \cdot)\|_{\cB_2(\cH_j,\cH)} \in L^2((a,b)), \; 
\|G_j (\cdot)\|_{\cB_2(\cH,\cH_j)} \in L^2((a,b)), \quad j=1,2.      \lb{A.78a}
\end{equation}
\end{hypothesis}

As an immediate consequence of Hypothesis \ref{hA.10} one infers the following facts.

\begin{lemma} \lb{lA.11}
Assume Hypothesis \ref{hA.10} and $\alpha \in \bbC$. Then, for a.e.\ $x \in (a,b)$, 
$\widehat F_j (x; \alpha) \in \cB_2(\cH_j,\cH)$, $\widehat F_j(\cdot\, ; \alpha)$ are 
$\cB_2(\cH_j,\cH)$-measurable, and 
\begin{equation}  
\big\|\widehat F_j( \cdot \, ; \alpha)\big\|_{\cB_2(\cH_j,\cH)} \in L^2((a,b)), \quad j=1,2.     \lb{A.84a}
\end{equation}
Moreover, 
\begin{align}
\begin{split} 
\int_c^d dx \, G_j (x) F_k(x), \, 
\int_c^d dx \, G_j (x) \hatt F_k(x; \alpha) \in \cB_1(\cH_k,\cH_j),&    \\
1 \leq j, k \leq 2, \; 
c, d \in (a,b) \cup \{a, b\},&      \lb{A.85a} 
\end{split} 
\end{align}
and 
\begin{align}
& QR, ST \in \cB_1\big(L^2((a,b);\cH)\big),    \lb{A.86b} \\
& \bsK, \bsH_a, \bsH_b \in \cB_2\big(L^2((a,b);\cH)\big).     \lb{A.86a} 
\end{align} 
Moreover,
\begin{equation}\lb{A.86c}
\begin{split}
\tr_{L^2((a,b);\cH)}(QR) = \tr_{\cH_2}(RQ) &= \int_a^b dx\, \tr_{\cH_2}(G_2(x)F_2(x)) \\
&= \int_a^bdx\, \tr_{\cH}(F_2(x)G_2(x)),
\end{split}
\end{equation}
and
\begin{equation}\lb{A.86d}
\begin{split}
\tr_{L^2((a,b);\cH)}(ST) = \tr_{\cH_1}(TS)&=\int_a^bdx\, \tr_{\cH_1}(G_1(x)F_1(x))\\
&=\int_a^bdx\, \tr_{\cH_1}(G_1(x)F_1(x)).
\end{split}
\end{equation}
\end{lemma}
\begin{proof}
As in the proof of Lemma \ref{lA.9}, one concludes that weak 
measurability of $\widehat F_j(\cdot\, ; \alpha)$, $j=1,2$, implies their $\cB_2(\cH_j,\cH)$-measurability 
by Pettis' theorem. The properties concerning $\widehat F_j( \cdot \, ; \alpha)$, $j=1,2$, then 
follow as in the proof of Lemma \ref{lA.6}, systematically replacing $\| \cdot \|_{\cB(\cH_j,\cH)}$ by 
$\| \cdot \|_{\cB_2(\cH_j,\cH)}$, $j=1,2$. 

Applying Lemma \ref{lA.9}, relations \eqref{A.85a} are now an immediate consequence 
of Hypothesis \ref{hA.10} and the fact that 
\begin{equation}
\big\|G_j (\cdot) \hatt F_k(\cdot \, ; \alpha)\big\|_{\cB_1(\cH_k,\cH_j)} \in L^1((a,b)), \quad 
1 \leq j, k \leq 2.     \lb{A.87a} 
\end{equation}

The proof of Lemma \ref{lA.9} (see \eqref{A.83a}) yields
\begin{align}
\begin{split} 
S \in \cB_2(\cH_1, L^2((a,b);\cH)), \, Q \in \cB_2(\cH_2, L^2((a,b);\cH)), \\  
T \in \cB_2(L^2((a,b);\cH), \cH_1), \, R \in \cB_2(L^2((a,b);\cH), \cH_2),   \lb{A.87A} 
\end{split} 
\end{align}
and \eqref{A.86b} follows.

Next, for any integral operator $\bsT$ in $L^2((a,b);\cH)$, with integral 
kernel satisfying $\|T(\cdot \, , \cdot \,)\|_{\cB_2(\cH)} \in L^2 ((a,b) \times (a,b); d^2x)$, one 
infers (cf.\ \cite[Theorem\ 11.6]{BS87}) that $\bsT \in \cB_2\big(L^2((a,b);\cH)\big)$ and 
\begin{equation}
\|\bsT\|_{\cB_2(L^2((a,b);\cH))} = 
\bigg(\int_a^b dx \int_a^b dx' \, \|T(x,x')\|_{\cB_2(\cH)}^2\bigg)^{1/2}.    \lb{A.82a} 
\end{equation}

Given Lemma \ref{lA.9} and the fact \eqref{A.82a}, one readily concludes \eqref{A.86a}.

Finally, the first equality in both \eqref{A.86c} (resp., \eqref{A.86d}) follows from cyclicity of the trace.  The other equalities throughout \eqref{A.86c} and \eqref{A.86d} follow from computing appropriate traces.  For example, taking an orthonormal basis $\{v_n\}_{n\in \bbN}$ in $\cH_2$, one computes
\begin{align}
&\tr_{\cH_2}(RQ)=\sum_{n\in \bbN} (v_n,RQv_n)_{\cH_2} = \sum_{n\in \bbN} \bigg(v_n, \int_a^bdx\, G_2(x)(Qv_n)(x)\bigg)_{\cH_2}\no\\
&\quad= \int_a^bdx\, \sum_{n\in \bbN} (v_n, G_2(x)(Qv_n)(x))_{\cH_2}= \int_a^bdx\, \sum_{n\in \bbN} (v_n,G_2(x)F_2(x)v_n)_{\cH_2}\no\\
&\quad= \int_a^bdx\, \tr_{\cH_2}(G_2(x)F_2(x)).\lb{A.82b}
\end{align}
\end{proof}

In the following we use many of the standard properties of Fredholm determinants, $2$-modified Fredholm determinants, and traces.  For the Fredholm determinant and trace,
\begin{equation}\lb{2.98z}
{\det}_{\cK}(I_{\cK}-A)=\prod_{n\in \cJ}(1-\lambda_n(A)),\quad A\in \cB_1(\cK),
\end{equation}
where $\{\lambda_n(A)\}_{n\in \cJ}$ is an enumeration of the non-zero eigenvalues of $A$, listed in non-increasing order according to their moduli, and $\cJ\subseteq \bbN$ is an appropriate indexing set. 
\begin{align}
& {\det}_{\cK}((I_\cK-A)(I_\cK-B))={\det}_{\cK}(I_\cK-A){\det}_{\cK}(I_\cK-B), \quad A, B
\in\cB_1(\cK), \lb{A.92a} \\   
& {\det}_{\cK}(I_{\cK}-AB)={\det}_{\cK'}(I_{\cK'}-BA), \quad {\tr}_{\cK} (AB) = {\tr}_{\cK'} (BA)   \lb{A.93a} \\
& \quad \text{ for all $A\in\cB_1(\cK',\cK)$, $B\in\cB(\cK,\cK')$ 
such that $AB\in\cB_1(\cK)$, $BA\in \cB_1(\cK')$,} \no
\intertext{and}
& {\det}_{\cK}(I_\cK-A)={\det}_{\cK_2}(I_{\cK_2}-D) \, \text{ for } \, A=\begin{pmatrix} 
0 & C \\ 0 & D \end{pmatrix}, \;\, D \in \cB_1(\cK_2), \; \cK=\cK_1\dotplus \cK_2, \lb{A.94a} 
\intertext{since}
&I_\cH -A=\begin{pmatrix} I_{\cK_1} & -C \\ 0 & I_{\cK_2}-D \end{pmatrix} =
\begin{pmatrix} I_{\cK_1} & 0 \\ 0 & I_{\cK_2}-D \end{pmatrix} 
\begin{pmatrix} I_{\cK_1} & -C \\ 0 & I_{\cK_2} \end{pmatrix}. \lb{A.95a} 
\end{align}

For $2$-modified Fredholm determinants,
\begin{equation}\lb{Z2.117}
{\det}_{2,\cK}(I_{\cK}-A) = \prod_{n\in\cJ}(1-\lambda_n(A))e^{\lambda_n(A)},\quad A\in \cB_2(\cK),
\end{equation}
where $\{\lambda_n(A)\}_{n\in \cJ}$ is an enumeration of the non-zero eigenvalues of $A$, listed in non-increasing order according to their moduli, and $\cJ\subseteq \bbN$ is an appropriate indexing set,
\begin{align}
& {\det}_{2,\cK}(I_{\cK}-A)= {\det}_{\cK}((I_{\cK}-A)\exp(A)), \quad A\in\cB_2 (\cK), \lb{Z3.21} \\
& {\det}_{2,\cK}((I_{\cK}-A)(I_{\cK}-B))={\det}_{2,\cK}(I_{\cK}-A){\det}_{2,\cK}(I_{\cK}-B)e^{-\tr_{\cK}(AB)},\lb{Z3.22}\\
&\hspace*{8.5cm}A, B\in\cB_2(\cK), \no \\
& {\det}_{2,\cK}(I_{\cK}-A)={\det}_{\cK}(I_{\cK}-A)e^{\tr_{\cK}(A)}, \quad A\in\cB_1(\cK). \lb{Z3.23}
\end{align}
Here $\cK$, $\cK'$, and $\cK_j$, $j=1,2$, are complex, separable Hilbert spaces,
$\cB(\cK)$ denotes the set of bounded linear operators on $\cK$, $\cB_p(\cK)$,
$p\geq 1$, denote the usual trace ideals of $\cB(\cK)$, and $I_\cK$ denotes
the identity operator in $\cK$. Moreover, ${\det}_{\cK}(I_\cK-A)$, $A\in\cB_1(\cK)$,
denotes the standard Fredholm determinant, with $\tr_{\cK}(A)$, $A\in\cB_1(\cK)$, the
corresponding trace, and $\det_{2,\cK}(I_{\cK}-A)$ the $2$-modified Fredholm determinant of a Hilbert--Schmidt operator $A\in \cB_2(\cK)$. Finally, $\dotplus$ in \eqref{A.94a} denotes a
direct, but not necessary orthogonal, sum decomposition of $\cK$ into $\cK_1$
and $\cK_2$. (We refer, e.g., to \cite{GGK96}, \cite{GGK97}, \cite[Ch.\ XIII]{GGK00}, 
\cite[Sects.\ IV.1~\&~IV.2]{GK69}, \cite[Ch.\ 17]{RS78}, \cite{Si77}, \cite[Ch.\ 3]{Si05} 
for these facts). 

\begin{theorem}\lb{tA.12}
Assume Hypothesis \ref{hA.10} and let $\alpha \in \bbC$.  Then
\begin{equation}\lb{Z2.121}
{\det}_{2,L^2((a,b);\cH)}(\bsI - \alpha \bsH_a) = {\det}_{2,L^2((a,b);\cH)}(\bsI - \alpha \bsH_b) = 1.
\end{equation}
Assume, in addition, that $U$ is given by \eqref{A.37}.  Then
\begin{align}
&{\det}_{2,L^2((a,b);\cH)}(\bsI - \alpha \bsK)   \no\\[1mm]
&\quad = {\det}_{\cH_1}\big(I_{\cH_1}-\alpha T(\bsI - \alpha \bsH_b)^{-1}S\big)
\exp\big(\alpha \tr_{L^2((a,b);\cH)}(ST)\big)\lb{Z2.122}\\[1mm]
&\quad = {\det}_{\cH_1}\bigg(I_{\cH_1}-\alpha \int_a^bdx\, G_1(x)\widehat F_1(x,\alpha)\bigg)
\exp\bigg(\alpha \int_a^bdx\, \tr_{\cH}(F_1(x)G_1(x))\bigg)\lb{Z2.123}\\[1mm]
&\quad = {\det}_{\cH_1\oplus \cH_2}(U(a,\alpha))
\exp\bigg(\alpha \int_a^bdx\, \tr_{\cH}(F_1(x)G_1(x))\bigg)\lb{Z2.124}   \\[1mm] 
&\quad = {\det}_{\cH_2}\big(I_{\cH_2} - \alpha R(\bsI - \alpha \bsH_a)^{-1}Q\big)
\exp\big(\alpha \tr_{L^2((a,b);\cH)}(QR) \big)\lb{Z2.125}\\[1mm] 
&\quad = {\det}_{\cH_2}\bigg(I_{\cH_2} - \alpha \int_a^bdx\, G_2(x)\widehat F_2(x,\alpha)\bigg)
\exp\bigg(\alpha \int_a^bdx\, \tr_{\cH}(F_2(x)G_2(x))\bigg)\lb{Z2.126}\\[1mm]
&\quad = {\det}_{\cH_1\oplus\cH_2}(U(b;\alpha))
\exp\bigg(\alpha \int_a^bdx\, \tr_{\cH}(F_2(x)G_2(x))\bigg).\lb{Z2.127}
\end{align}
\end{theorem}
\begin{proof}
Since $\bsH_a$ and $\bsH_b$ are quasi-nilpotent, they have no non-zero eigenvalues.  Therefore, \eqref{Z2.121} follows from the representation of the $2$-modified Fredholm determinant given in \eqref{Z2.117}.

Next, one observes
\begin{align}
\bsI-\alpha \bsK &=(\bsI-\alpha \bsH_a)[\bsI-\alpha (\bsI-\alpha \bsH_a)^{-1}QR] \lb{A.107a} \\ 
&=(\bsI-\alpha \bsH_b)[\bsI-\alpha (\bsI-\alpha \bsH_b)^{-1}ST]. \lb{A.108a}
\end{align}
Using the various properties of determinants given in \eqref{Z3.21}--\eqref{Z3.23} and \eqref{A.108a}, one computes
\begin{align}
&{\det}_{2,L^2((a,b);\cH)}(\bsI - \alpha \bsK)\no\\
&\quad = {\det}_{2,L^2((a,b);\cH)}\big((\bsI - \alpha \bsH_b)
\big[\bsI - \alpha (\bsI - \bsH_b)^{-1}ST\big]\big)\no\\
&\quad = {\det}_{2,L^2((a,b);\cH)}(\bsI-\alpha \bsH_b) \, {\det}_{2,L^2((a,b);\cH)}\big(\bsI-\alpha (\bsI-\bsH_b)^{-1}ST\big)\no\\
&\qquad \times \exp\big(-\tr_{L^2((a,b);\cH)}\big(\alpha^2\bsH_b(\bsI - \bsH_b)^{-1}ST\big)\big)\no\\
&\quad = {\det}_{L^2((a,b);\cH)}\big(\bsI - \alpha (\bsI - \bsH_b)^{-1}ST\big)\exp\big(\tr_{L^2((a,b);\cH)}\big(\alpha (\bsI - \bsH_b)^{-1}ST\big) \big)\no\\
&\qquad \times\exp\big(-\tr_{L^2((a,b);\cH)}\big(\alpha^2\bsH_b(\bsI - \bsH_b)^{-1}ST\big)\big)\lb{Z2.128}\\
&\quad = {\det}_{L^2((a,b);\cH)}\big(\bsI - \alpha (\bsI - \bsH_b)^{-1}ST\big)\exp\big(\alpha\tr_{L^2((a,b);\cH)}(ST) \big)\no\\
&\quad = {\det}_{\cH_1}\big(I_{\cH_1} - \alpha T(\bsI - \bsH_b)^{-1}S\big)\exp\big(\alpha \tr_{L^2((a,b);\cH)}(ST) \big)\lb{Z2.129}\\
&\quad = {\det}_{\cH_1}\bigg( I_{\cH_1} -\alpha \int_a^bdx\, G_1(x)\widehat F_1(x;\alpha)\bigg)\exp\bigg(\alpha \int_a^bdx\, \tr_{\cH}(F_1(x)G_1(x)) \bigg)\no\\
&\quad = {\det}_{\cH_1\oplus \cH_2}(U(a;\alpha))\exp\bigg(\alpha \int_a^bdx\, \tr_{\cH}(F_1(x)G_1(x)) \bigg).\no
\end{align}
In the above calculation, \eqref{Z2.128} is an application of \eqref{Z3.23}, noting that $ST\in \cB_1(L^2((a,b);\cH))$ by Lemma \ref{lA.11}, while \eqref{Z2.129} makes use of the determinant property in \eqref{A.93a}.

To prove ${\det}_{2,L^2((a,b);\cH)}(\bsI - \alpha \bsK)$ coincides with the expressions in \eqref{Z2.125}--\eqref{Z2.127}, we apply \eqref{A.107a} and carry out the analogous computation,
\begin{align}
&{\det}_{2,L^2((a,b);\cH)}(\bsI - \alpha \bsK)\no\\
&\quad = {\det}_{2,L^2((a,b);\cH)}\big((\bsI - \alpha \bsH_a)\big[\bsI - \alpha (\bsI - \bsH_a)^{-1}QR\big]\big)\no\\
&\quad = {\det}_{2,L^2((a,b);\cH)}(\bsI-\alpha \bsH_a) \, 
{\det}_{2,L^2((a,b);\cH)}\big(\bsI-\alpha (\bsI-\bsH_a)^{-1}QR\big)\no\\
&\qquad \times \exp\big(-\tr_{L^2((a,b);\cH)}\big(\alpha^2\bsH_a(\bsI - \bsH_a)^{-1}QR\big)\big)\no\\
&\quad = {\det}_{L^2((a,b);\cH)}\big(\bsI - \alpha (\bsI - \bsH_a)^{-1}QR\big)\exp\big(\tr_{L^2((a,b);\cH)}\big(\alpha (\bsI - \bsH_a)^{-1}QR\big)\big)\no\\
&\qquad \times\exp\big(-\tr_{L^2((a,b);\cH)}\big(\alpha^2\bsH_a(\bsI - \bsH_a)^{-1}QR\big)\big)\no\\
&\quad = {\det}_{L^2((a,b);\cH)}\big(\bsI - \alpha (\bsI - \bsH_a)^{-1}QR\big)\exp\big(\alpha\tr_{L^2((a,b);\cH)}(QR) \big)\no\\
&\quad = {\det}_{\cH_2}\big(I_{\cH_2} - \alpha R(\bsI - \bsH_a)^{-1}Q\big)\exp\big(\alpha \tr_{L^2((a,b);\cH)}(QR) \big)\no\\
&\quad = {\det}_{\cH_2}\bigg( I_{\cH_2} -\alpha \int_a^bdx\, G_2(x)\widehat F_2(x;\alpha)\bigg)\exp\bigg(\alpha \int_a^bdx\, \tr_{\cH}(F_2(x)G_2(x)) \bigg)\no\\
&\quad = {\det}_{\cH_1\oplus \cH_2}(U(b;\alpha))\exp\bigg(\alpha \int_a^bdx\, \tr_{\cH}(F_2(x)G_2(x)) \bigg).\lb{Z2.130}
\end{align}
\end{proof}

As a consequence of Theorem \ref{tA.12}, we recover the following result in the case where $\bsK$ is trace class $\bsK\in \cB_1(L^2((a,b);\cH))$, not just $\bsK\in \cB_2(L^2((a,b);\cH))$, first proved in \cite{CGPST13}.
\begin{corollary} [\cite{CGPST13}] \lb{cA.13} 
Assume Hypothesis \ref{hA.10}, let $\alpha\in\bbC$, and suppose that $\bsK$ belongs to the trace class, $\bsK\in \cB_1(L^2((a,b);\cH))$. Then, $\bsH_a,\bsH_b\in \cB_1(L^2((a,b);\cH))$ and
\begin{align} 
&{\tr}_{L^2((a,b);\cH)}(\bsH_a)={\tr}_{L^2((a,b);\cH)}(\bsH_b)=0,     \lb{A.96a} \\ 
&{\det}_{L^2((a,b);\cH)}(\bsI-\alpha \bsH_a)
= {\det}_{L^2((a,b);\cH)}(\bsI-\alpha \bsH_b)=1, 
\lb{A.97a} \\ 
&\tr_{L^2((a,b);\cH)}(\bsK)=\int_a^b dx\,\tr_{\cH_1}(G_1(x)F_1(x))
=\int_a^b dx\,\tr_{\cH}(F_1(x)G_1(x)) \lb{A.98a} \\
& \quad =\int_a^b dx\,\tr_{\cH_2}(G_2(x)F_2(x))
=\int_a^b dx\,\tr_{\cH}(F_2(x)G_2(x)). \lb{A.99a} 
\end{align}
Assume in addition that $U$ is given by \eqref{A.37}. Then,
\begin{align}
&{\det}_{L^2((a,b);\cH)}(\bsI-\alpha \bsK)={\det}_{\cH_1}\big(I_{\cH_1}
-\alpha T(\bsI - \alpha \bsH_b)^{-1}S\big)   \lb{A.100a} \\ 
& \quad ={\det}_{\cH_1}\bigg(I_{\cH_1}-\alpha \int_a^b dx\,
G_1(x)\widehat F_1(x; \alpha )\bigg)    \lb{A.101a} \\
& \quad ={\det}_{\cH_1 \oplus \cH_2}(U(a; \alpha)) \lb{A.102a} \\
& \quad ={\det}_{\cH_2}\big(I_{\cH_2}-\alpha R(\bsI - \alpha \bsH_a)^{-1}Q\big)
\lb{A.103a} \\
& \quad ={\det}_{\cH_2}\bigg(I_{\cH_2}-\alpha \int_a^b dx\,
G_2(x)\widehat F_2(x; \alpha )\bigg) \lb{A.104a} \\
& \quad ={\det}_{\cH_1 \oplus \cH_2}(U(b; \alpha)).    \lb{A.105a} 
\end{align}
\end{corollary}
\begin{proof}
If $\bsK\in \cB_1(L^2((a,b);\cH))$, then $\bsH_a,\bsH_b\in \cB_1(L^2((a,b);\cH))$ is a consequence of \eqref{A.14} and \eqref{A.15}, since $QR,ST\in \cB_1(L^2((a,b);\cH))$ by Lemma \ref{lA.11} (cf.~\eqref{A.86b}).
Since $\bsH_a$ and $\bsH_b$ are quasi-nilpotent, they have no non-zero eigenvalues.  Thus, relations \eqref{A.96a} are clear from Lidskii's theorem (cf., e.g., \cite[Theorem\ VII.6.1]{GGK90}, \cite[Sect.\ III.8, Sect.\ IV.1]{GK69}, \cite[Theorem\ 3.7]{Si05}), and the relations \eqref{A.97a} follow from \eqref{2.98z}.  Subsequently, \eqref{A.14}, \eqref{A.15}, and cyclicity of the trace (i.e., the second equality in \eqref{A.93a}) imply
\begin{align}
\begin{split} 
\tr_{L^2((a,b);\cH)}(\bsK) &= \tr_{L^2((a,b);\cH)}(QR)=\tr_{\cH_2}(RQ)  \\
&= \tr_{L^2((a,b);\cH)}(ST)=\tr_{\cH_1}(TS). \lb{A.106a}
\end{split} 
\end{align}
The equalities throughout \eqref{A.98a} and \eqref{A.99a} then follow from \eqref{A.86c} and \eqref{A.86d}.  Finally, relations \eqref{A.100a}--\eqref{A.105a} follow from those throughout \eqref{Z2.130}, \eqref{A.98a}, and \eqref{A.99a}, noting that (cf.~\eqref{Z3.23})
\begin{equation}\lb{2.130}
{\det}_{L^2((a,b);\cH)}(\bsI - \alpha\bsK) = {\det}_{2,L^2((a,b);\cH)}(\bsI-\alpha\bsK)\exp(-\alpha \tr_{L^2((a,b);\cH)}(\bsK)).
\end{equation}
\end{proof}

The results \eqref{A.96a}--\eqref{A.100a}, \eqref{A.102a}, \eqref{A.103a}, 
\eqref{A.105a} can be found in the finite-dimensional context 
($\dim(\cH) < \infty$ and $\dim(\cH_j) < \infty$, $j=1,2$) in 
Gohberg, Goldberg, and Kaashoek \cite[Theorem 3.2]{GGK90} and in 
Gohberg, Goldberg, and Krupnik \cite[Sects.\ XIII.5, XIII.6]{GGK00} under the 
additional assumptions that $a,b$ are finite. The more general case where 
$(a,b)\subseteq\bbR$ is an arbitrary interval, as well as \eqref{A.101a} and 
\eqref{A.104a}, still in the case where $\cH$ and $\cH_j$, $j=1,2$, are 
finite-dimensional, was derived in \cite{GM03}.

\section{Some Applications to Schr\"odinger Operators with Operator-Valued Potentials} \lb{s3}

To illustrate the potential of the theory developed in Section \ref{s2}, we now briefly 
discuss some applications to Schr\"odinger operators with operator-valued potentials.

We start with some necessary notation:
Let $(a,b) \subseteq \bbR$ be a finite or infinite interval and $\cH$ a complex, separable 
Hilbert space. Integration of $\cH$-valued functions on $(a,b)$ will
always be understood in the sense of Bochner, in particular, if $p\ge 1$, the Banach space 
$L^p((a,b);dx;\cH)$ denotes the set of equivalence classes of strongly measurable $\cH$-valued 
functions which differ at 
most on sets of Lebesgue measure zero, such that $\|f(\cdot)\|_{\cH}^p \in L^1((a,b);dx)$. The
corresponding norm in $L^p((a,b);dx;\cH)$ is given by
\begin{equation}
\|f\|_{L^p((a,b);dx;\cH)} = \bigg(\int_{(a,b)} dx\, \|f(x)\|_{\cH}^p \bigg)^{1/p}. 
\end{equation}
In the case $p=2$, $L^2((a,b);dx;\cH)$ is a separable Hilbert space. 
One recalls that by a result of Pettis \cite{Pe38}, weak 
measurability of $\cH$-valued functions implies their strong measurability.

Sobolev spaces $W^{n,p}((a,b); dx; \cH)$ for $n\in\bbN$ and $p\geq 1$ are defined as follows: 
$W^{1,p}((a,b);dx;\cH)$ is the set of all
$f\in L^p((a,b);dx;\cH)$ such that there exists a $g\in L^p((a,b);dx;\cH)$ and an
$x_0\in(a,b)$ such that
\begin{equation}
f(x)=f(x_0)+\int_{x_0}^x dx' \, g(x') \, \text{ for a.e.\ $x \in (a,b)$.}
\end{equation}
In this case $g$ is the strong derivative of $f$, $g=f'$. Similarly,
$W^{n,p}((a,b);dx;\cH)$ is the set of all $f\in L^p((a,b);dx;\cH)$ so that the first $n$ strong
derivatives of $f$ are in $L^p((a,b);dx;\cH)$. 

For simplicity of notation, from this point on we will omit the Lebesgue measure whenever 
no confusion can occur and henceforth simply write $L^p ((a,b);\cH)$ for 
$L^p((a,b);dx;\cH)$. Moreover, in the special case where $\cH = \bbC$, we omit $\cH$ 
and typically (but not always) the Lebesgue measure and just write $L^p((a,b))$.

We begin with some applications recently considered in \cite{CGPST13} which illustrate 
Theorem \ref{tA.12} and Corollary \ref{cA.13}. We closely follow the treatment in \cite{CGPST13} 
and refer to \cite{GWZ13}, \cite{GWZ13a} for background on Schr\"odinger operators with 
operator-valued potentials. 

We start with the following basic assumptions. 

\begin{hypothesis} \lb{hB.1}
Suppose that $V:\bbR \to \cB_1(\cH)$ is a weakly
measurable operator-valued function with $\|V(\cdot)\|_{\cB_1(\cH)}\in L^1(\bbR)$.
\end{hypothesis}

We note that no self-adjointness condition $V(x) = V(x)^*$ for a.e.\ $x \in \bbR$ is assumed 
to hold in $\cH$.

We introduce the densely defined, closed, linear operators in $L^2(\bbR;\cH)$ defined by 
\begin{align}
&\bsH_0 f=-f'', \quad f\in \dom(\bsH_0) = W^{2,2}(\bbR;\cH), \lb{B.1} \\
&\bsH f=\tau f, \lb{B.2} \\
&f\in\dom(\bsH)=\{g\in L^2(\bbR;\cH) \,|\, g,g' \in AC_{\loc}(\bbR; \cH); \, \tau g
\in L^2(\bbR;\cH)\},     \no
\end{align}
where we denoted 
\begin{equation} 
(\tau f)(x) = -f''(x) + V(x) f(x) \, \text{ for a.e.\ $x\in\bbR$.}
\end{equation}  
In addition, we introduce the densely defined, closed, linear operator $\bsV$ in $L^2(\bbR;\cH)$ by 
\begin{align}
& (\bsV f)(x) = V(x) f(x),    \no \\
& f \in \dom(\bsV) = \bigg\{g \in L^2(\bbR;\cH) \, \bigg| \, g(x) \in \dom(V(x)) 
\text{ for a.e.\ $x \in \bbR$,}     \lb{B.2a} \\
& \hspace*{2cm} x \mapsto V(x) g(x) \text{ is (weakly) measurable,} \, 
\int_{\bbR} dx \, \|V(x) g(x)\|^2_{\cH} < \infty\bigg\}.      \no
\end{align}

Next we turn to the $\cB(\cH)$-valued Jost solutions $f_\pm (z,\cdot)$ of
\begin{equation} 
-\psi''(z,x)+V(x)\psi(z,x)=z\psi(z,x), \quad z\in\bbC, \; x \in \bbR,    \lb{B.2b} 
\end{equation}
(i.e., $f_\pm (z,\cdot) h\in W^{2,1}_\loc((a,b);\cH)$ for every $h\in\cH$) defined by 
\begin{align}
\begin{split}
f_\pm (z,x)=e^{\pm iz^{1/2}x} I_{\cH} - \int_x^{\pm\infty} dx' \,
g_0(z,x,x')V(x')f_\pm (z,x'),& \lb{B.3} \\
z \in \bbC, \; \Im(z^{1/2})\geq 0, \; x\in\bbR,&    
\end{split} 
\end{align}
where $g_0(z,\cdot,\cdot)$ is the $\cB(\cH)$-valued Volterra Green's function of 
$\bsH_0$ given by 
\begin{equation}
g_0(z,x,x') = z^{-1/2}\sin(z^{1/2}(x-x')) I_{\cH}, \quad z \in \bbC, \; x, x' \in \bbR.       \lb{B.4}
\end{equation}

We also recall the $\cB(\cH)$-valued Green's function of $\bsH_0$,
\begin{align}
\begin{split} 
G_0(z,x,x')=\big(\bsH_0 - z \bsI\big)^{-1}(x,x')=
\f{i}{2z^{1/2}}e^{iz^{1/2}|x-x'|} I_{\cH},&     \lb{B.5} \\ 
z \in \bbC \backslash [0,\infty), \; \Im(z^{1/2})>0, \; x,x'\in\bbR,&
\end{split} 
\end{align}
with $\bsI$ representing the identity operator in $L^2(\bbR;\cH)$. 

The $\cB(\cH)$-valued Jost function $\cF$ associated with the pair of self-adjoint 
operators $(\bsH, \bsH_0)$ is then given by
\begin{align}
\cF(z)&=\f{1}{2iz^{1/2}} W(f_-(\ol z)^*,f_+(z))  \lb{B.7} \\
&=I_{\cH}-\f{1}{2iz^{1/2}}\int_\bbR dx \, e^{- i z^{1/2}x}V(x)f_+ (z,x),    \lb{B.8} \\
&=I_{\cH}-\f{1}{2iz^{1/2}}\int_\bbR dx \, f_- (\ol z,x)^* V(x) e^{i z^{1/2}x},       \lb{B.8a} \\
& \hspace*{2.6cm} z \in \bbC \backslash \{0\}, \; \Im(z^{1/2})\geq 0.    \no 
\end{align}
Here $W(\cdot, \cdot)$ denotes the Wronskian defined by 
\begin{equation}
W(F_1,F_2)(x) = F_1(x) F'_2(x) - F'_1(x) F_2(x), \quad x \in (a,b),     
\end{equation}
for $F_1,F_2$ strongly continuously differentiable $\cB(\cH)$-valued
functions.

Next, we recall the polar decomposition of a densely defined, closed, linear operator $S$ 
in a complex separable Hilbert space $\cK$
\begin{equation}
S = |S| U_S = U_S |S|,    \lb{B.11} 
\end{equation}
where $U_S$ is a partial isometry in $\cK$ and $|S| = (S^* S)^{1/2}$,  

Introducing the factorization of $\bsV = \bsu \bsv$, where 
\begin{equation}
\bsu = |\bsV|^{1/2} \bsU_{\bsV} = \bsU_{\bsV} |\bsV|^{1/2}, \quad  \bsv = |\bsV|^{1/2}, 
\quad \bsV = |\bsV| \bsU_{\bsV} = \bsU_{\bsV} |\bsV| = \bsu \bsv = \bsv \bsu,    \lb{B.12} 
\end{equation}
one verifies one verifies (see, e.g., \cite{GLMZ05}, \cite{Ka66} and the references cited therein) that 
\begin{align}
& (\bsH - z \bsI)^{-1} - (\bsH_0 - z \bsI)^{-1}    \lb{B.13} \\ 
& \quad = (\bsH_0 - z \bsI)^{-1} \bsv \big[\bsI + \ol{\bsu (\bsH_0 - z \bsI)^{-1} \bsv}\big]^{-1} 
\bsu (\bsH_0 - z \bsI)^{-1}, \quad z\in\bbC\backslash\sigma(\bsH). \no
\end{align}

Next, to make contact with the notation used in Section \ref{s2}, we now 
introduce the operator $\bsK(z)$ in $L^2(\bbR;\cH)$ by
\begin{equation}
\bsK(z)=-\ol{\bsu (\bsH_0 - z \bsI)^{-1}\bsv}, \quad
z\in\bbC\backslash [0,\infty),   \lb{B.14}
\end{equation}
with integral kernel
\begin{equation}
K(z,x,x')=-u(x)G_0(z,x,x')v(x'), \quad z \in \bbC \backslash [0,\infty), \; \Im(z^{1/2}) > 0, 
\; x,x' \in\bbR,  \lb{B.15} 
\end{equation}
and the Volterra operators $\bsH_{-\infty} (z)$, $\bsH_\infty (z)$ (cf.\ \eqref{A.4},
\eqref{A.12a}) in $L^2(\bbR;\cH)$, with integral kernel
\begin{equation}
H(z,x,x')=u(x)g^{(0)}(z,x,x')v(x').  \lb{B.16} 
\end{equation}
Here we used the abbreviations,  
\begin{align}
\begin{split} 
& u(x) = |V(x)|^{1/2} U_{V(x)}, \quad v(x) = |V(x)|^{1/2},    \lb{B.17} \\ 
& V(x) = |V(x)| U_{V(x)} = U_{V(x)} |V(x)| = u(x) v(x) \, \text{ for a.e.\ $x\in\bbR$.} 
\end{split} 
\end{align}

Moreover, we introduce for a.e.\ $x\in\bbR$, 
\begin{align}
\begin{split}
f_1(z,x)&=-u(x) e^{iz^{1/2}x}, \hspace*{.68cm} 
g_1(z,x)=(i/2)z^{-1/2}v(x)e^{-iz^{1/2}x},  \\
f_2(z,x)&=-u(x)e^{-iz^{1/2}x}, \quad \; 
g_2(z,x)=(i/2)z^{-1/2}v(x)e^{iz^{1/2}x}. \lb{B.18}
\end{split}
\end{align}
Assuming temporarily that 
\begin{equation}
\supp(\|V(\cdot)\|_{\cB(\cH)}) \text{ is compact} \lb{B.19}
\end{equation}
(employing the notion of support for regular distributions on $\bbR$) 
in addition to Hypothesis \ref{hB.1}, identifying $\cH_1 = \cH_2 = \cH$, and introducing 
$\hat f_j(z,\cdot)$, $j=1,2$, by
\begin{align}
\hat f_1(z,x)&=f_1(z,x)-\int_x^\infty dx'\, H(z,x,x')\hat f_1(z,x'), \lb{B.20}
\\ 
\hat f_2(z,x)&=f_2(z,x)+\int_{-\infty}^x dx'\, H(z,x,x')\hat f_2(z,x'),
\lb{B.21} \\ 
& \hspace*{1.95mm} z\in\bbC\backslash [0,\infty), \; \Im(z^{1/2}) > 0, \; \text{a.e.\ } \, x \in\bbR, \no
\end{align}  
yields $\hat f_j(z,\cdot) \in L^2(\bbR;\cH)$, $j=1,2$, upon a standard iteration of the Volterra 
integral equations \eqref{B.20}, \eqref{B.21}. In fact, $\hat f_j(z,\cdot) \in L^2(\bbR;\cH)$, $j=1,2$, have 
compact support as long as \eqref{B.19} holds. By
comparison with \eqref{B.3}, one then identifies for all $z\in\bbC\backslash [0,\infty)$, 
$\Im(z^{1/2}) > 0$, and a.e.\ $x \in\bbR$, 
\begin{equation}
\hat f_1(z,x)=-u(x) f_+ (z,x),  \quad  \hat f_2(z,x)=-u(x) f_- (z,x). \lb{B.23} 
\end{equation}
We note that the temporary compact support assumption \eqref{B.19} on $\|V(\cdot)\|_{\cB(\cH)}$ has
only been introduced to guarantee that $f_j(z,\cdot), \hat f_j(z,\cdot) \in
L^2(\bbR;\cH)$, $j=1,2$ for all $z\in\bbC\backslash [0,\infty)$, 
$\Im(z^{1/2}) > 0$. This extra hypothesis can be removed by a standard approximation argument 
(see, \cite{CGPST13}, \cite{GM03}). 

Recalling the following basic fact (cf.\ \cite{CGPST13}), 
\begin{equation}
\bsK(z)\in \cB_1\big(L^2(\bbR;\cH)\big), \quad z\in\bbC\backslash [0,\infty), \lb{B.24}
\end{equation}
still assuming Hypothesis \ref{hB.1}, an application of Lemma \ref{lA.7} and Corollary \ref{cA.13} 
then yields the following Fredholm determinant reduction result, identifying the Fredholm determinant 
of $\bsI - \bsK(z)$ and that of the $\cB(\cH)$-valued Jost function $\cF(z)$ (the inverse transmission 
coefficient).

\begin{theorem} [\cite{CGPST13}] \lb{tB.2}
Assume Hypothesis \ref{hB.1}, then 
\begin{equation}
{\det}_{L^2(\bbR;\cH)}(\bsI - \bsK(z)) = {\det}_{\cH}(\cF(z)), \quad z\in\bbC\backslash [0,\infty).   \lb{B.26}
\end{equation} 
\end{theorem}

Relation \eqref{B.26} represents the infinite-dimensional version of the celebrated Jost--Pais-type 
reduction of Fredholm determinants \cite{JP51} (see also \cite{Ge86}, \cite{GM03}, \cite{Ne80}, 
\cite{Si00}). 

Next, we revisit the second-order equation \eqref{B.2b} from a different perspective. We intend to 
rederive the result analogous to \eqref{B.26} in the context of $2$-modified determinants 
$\det_2(\cdot)$ by rewriting the second-order Schr\"odinger equation as a first-order
$2\times 2$ block operator system, taking the latter as our point of departure. (In the special case 
where $\cH$ is finite-dimensional, this was considered in \cite{GLM07}, \cite{GLZ08}, 
\cite{GM03}, \cite{KM14}.) 

Assuming Hypothesis \ref{hB.1} for the rest of this example, the Schr\"odinger
equation with the operator-valued potential $V(\cdot)$, 
\begin{equation}
-\psi''(z,x) + V(x)\psi(z,x)=z\psi(z,x), \lb{B.27}
\end{equation} 
is equivalent to the first-order $2 \times 2$ block operator system 
\begin{equation}
\Psi'(z,x)=\begin{pmatrix} 0 & I_{\cH} \\ V(x)-z & 0 \end{pmatrix} \Psi(z,x), \quad 
\Psi(z,x) = \begin{pmatrix} \psi(z,x) \\ \psi'(z,x) \end{pmatrix}. \lb{B.28}
\end{equation}
Since $\Phi^{(0)}$ defined by 
\begin{align}
\begin{split} 
\Phi^{(0)} (z,x)= \f{1}{2i z^{1/2}}\begin{pmatrix} \exp(-iz^{1/2}x) I_{\cH} & \exp(iz^{1/2}x) I_{\cH} \\
-iz^{1/2} \exp(-iz^{1/2}x) I_{\cH} & iz^{1/2} \exp(iz^{1/2}x) I_{\cH} \end{pmatrix},& \\
\Im(z^{1/2})\geq 0,&    \lb{B.29}
\end{split} 
\end{align}
is a fundamental block operator matrix of the system \eqref{B.28} in the case $V=0$ a.e., and
since 
\begin{equation}
\Phi^{(0)} (z,x) \Phi^{(0)} (z,x')^{-1}=\begin{pmatrix} 
\cos(z^{1/2}(x-x')) I_{\cH} & z^{-1/2}\sin(z^{1/2}(x-x')) I_{\cH} \\
-z^{1/2}\sin(z^{1/2}(x-x')) I_{\cH} &  \cos(z^{1/2}(x-x')) I_{\cH} \end{pmatrix}, \lb{B.30}
\end{equation} 
the system \eqref{B.28} has the following pair of linearly independent
solutions for $z\neq 0$,
\begin{align}
&F_\pm(z,x)=F^{(0)}_\pm (z,x) \no \\
& \quad - \int_x^{\pm\infty} dx' \begin{pmatrix} 
\cos(z^{1/2}(x-x')) I_{\cH} & z^{-1/2}\sin(z^{1/2}(x-x')) I_{\cH} \\
-z^{1/2}\sin(z^{1/2}(x-x')) I_{\cH} &  \cos(z^{1/2}(x-x')) I_{\cH} \end{pmatrix} \no \\
& \hspace{2.3cm} \times \begin{pmatrix} 0 & 0 \\ V(x') & 0 \end{pmatrix}
F_\pm(z,x') \no \\ 
& \quad =F^{(0)}_\pm (z,x) - \int_x^{\pm\infty} dx' 
\begin{pmatrix} z^{-1/2}\sin(z^{1/2}(x-x')) I_{\cH} & 0 \\
\cos(z^{1/2}(x-x')) I_{\cH} & 0 \end{pmatrix} V(x') F_\pm(z,x'),    \no \\
& \hspace*{6.5cm} \Im(z^{1/2})\geq 0, \; z \neq 0, \; x\in\bbR,    \lb{B.31}
\end{align}
where we abbreviated
\begin{equation}
F^{(0)}_\pm (z,x) = \begin{pmatrix} I_{\cH} \\ \pm iz^{1/2} I_{\cH} \end{pmatrix} 
\exp(\pm i z^{1/2} x). \lb{B.32}
\end{equation}
By inspection, one has
\begin{equation}
F_\pm (z,x)=\begin{pmatrix} f_\pm (z,x) \\ f'_\pm(z,x) \end{pmatrix}, \quad 
\Im(z^{1/2})\geq 0, \; z \neq 0, \; x\in\bbR,   \lb{B.33}
\end{equation}
with $f_\pm (z,\cdot)$ given by \eqref{B.3}. Next, one introduces 
\begin{align}
\begin{split} 
f_1(z,x)&=-u(x) \begin{pmatrix} I_{\cH} \\ i z^{1/2} I_{\cH} \end{pmatrix} 
\exp( i z^{1/2} x),  \\
f_2(z,x)&= -u(x) \begin{pmatrix} I_{\cH} \\ -i z^{1/2} I_{\cH} \end{pmatrix} \exp( -i z^{1/2} x),     \\
g_1(z,x)&=v(x)\bigg(\f{i}{2z^{1/2}} \exp(-iz^{1/2} x) I_{\cH} \quad 0\bigg),     \\ 
g_2(z,x)&=v(x)\bigg(\f{i}{2z^{1/2}} \exp(iz^{1/2} x) I_{\cH} \quad 0\bigg),     \lb{B.34} 
\end{split} 
\end{align}
and hence
\begin{align}
\begin{split} 
H(z,x,x')& = f_1(z,x) g_1(z,x')-f_2(z,x)g_2(z,x') \lb{B.35} \\
& = u(x) \begin{pmatrix} z^{-1/2}\sin(z^{1/2}(x-x')) I_{\cH} & 0 \\
\cos(z^{1/2}(x-x')) I_{\cH} & 0 \end{pmatrix} v(x') 
\end{split} 
\end{align}
and we introduce
\begin{align}
\widetilde K(z,x,x')&=\begin{cases} f_1(z,x)g_1(z,x'), & x'<x, \\ 
f_2(z,x)g_2(z,x'), & x<x',\end{cases}  \lb{B.36} \\
&=\begin{cases} -u(x)\f{1}{2}\exp(iz^{1/2}(x-x'))\begin{pmatrix} iz^{-1/2} I_{\cH} & 0
\\  - I_{\cH} & 0  \end{pmatrix} v(x'), & x'<x, \\ 
-u(x)\f{1}{2}\exp(-iz^{1/2}(x-x'))\begin{pmatrix} iz^{-1/2} I_{\cH} & 0 \\  
I_{\cH} & 0  \end{pmatrix} v(x'), & x<x', \end{cases}    \no \\
& \hspace*{4.5cm} \Im(z^{1/2})\geq 0, \, z\neq 0, \; x,x'\in\bbR.    \lb{B.37} 
\end{align}
One notes that $\widetilde K(z,\cdot,\cdot)$ is discontinuous on the diagonal
$x=x'$.  Since 
\begin{equation}
\widetilde K(z,\cdot,\cdot)\in L^2(\bbR^2;dx\,dx'; \cH)^{2\times 2}, \quad \Im(z^{1/2})\geq 0,
\, z\neq 0,  \lb{B.38}
\end{equation}
the associated operator $\widetilde \bsK (z)$ with integral kernel \eqref{B.37} is
Hilbert--Schmidt,
\begin{equation}
\widetilde \bsK (z) \in \cB_2 \big(L^2(\bbR; \cH)^2\big), \quad \Im(z^{1/2})\geq 0, \; 
z\neq 0.  \lb{B.39}
\end{equation}
Next, assuming again temporarily \eqref{B.19}, the integral equations defining $\hat f_j(z,x)$, $j=1,2$,  
\begin{align}
\hat f_1(z,x)&=f_1(z,x)-\int_x^\infty dx'\, H(z,x,x')\hat f_1(z,x'), 
\lb{B.40} \\ 
\hat f_2(z,x)&=f_2(z,x)+\int_{-\infty}^x dx'\, H(z,x,x')\hat f_2(z,x'),
\lb{B.41} \\ 
&\hspace*{2cm}  \Im(z^{1/2})\geq 0, \; z\neq 0, \; x \in\bbR, \no
\end{align}  
yield solutions $\hat f_j(z,\cdot) \in L^2(\bbR; \cH)^2$, $j=1,2$. By
comparison with \eqref{B.31}, one then identifies
\begin{equation}
\hat f_1(z,x) = - u(x) F_+ (z,x),  \quad \hat f_2(z,x) = - u(x) F_- (z,x).      \lb{B.42}
\end{equation} 
We note that the temporary compact support assumption on $V$ has
only been invoked to guarantee that $f_j(z,\cdot), \, \hat f_j(z,\cdot) \in
L^2(\bbR; \cH)^2$, $j=1,2$. This extra hypothesis can be removed along a standard 
approximation method as detailed in \cite{GM03}. 

An application of Lemma \ref{lA.11} and Theorem \ref{tA.12} then yields the
following result (with $\bsK(\cdot)$ defined in \eqref{B.14} and $\bsI_2$ denoting the unit operator 
in $L^2(\bbR;\cH)^2$). 

\begin{theorem} \lb{tB.3}
Assume Hypothesis \ref{hB.1}, then  
\begin{align}
{\det}_{2, L^2(\bbR;\cH)^2} \big(\bsI_2 - \widetilde \bsK (z)\big) 
&=\cF(z)\exp\bigg(-\f{i}{2z^{1/2}} \int_\bbR dx\,
\tr_{\cH}(V(x))\bigg) \lb{B.43} \\ 
&={\det}_{2, L^2(\bbR;\cH)} (\bsI - \bsK(z)), \quad z \in \bbC \backslash [0,\infty).   \lb{B.44}
\end{align} 
\end{theorem}

Thus, equation \eqref{B.27} and the first-order system \eqref{B.28} share the same  
$2$-modified Fredholm determinant. 

While we focused on Schr\"odinger operators and associated first-order systems with operator-valued potentials on $\bbR$, completely analogous results can be derived on the half-line $(0,\infty)$. Rather 
than repeating such applications for the half-line, we turn to a slightly different application involving semi-separable integral operators in $L^2((0,\infty); \cH)$ analogous to \eqref{B.14}.  

We introduce the following basic assumptions. 

\begin{hypothesis} \lb{hB.4}
Let $V: (0,\infty) \to \cB_1(\cH)$ be a weakly
measurable operator-valued function with $\|V(\cdot)\|_{\cB_1(\cH)}\in L^1((0,\infty); (1+x)dx)$.
\end{hypothesis}

Again we note that $V(x)$ is not necessarily assumed to be self-adjoint in $\cH$ for 
a.e.\ $x \geq 0$.

In analogy to \eqref{B.1}, \eqref{B.2}, we introduce the densely defined, closed, Dirichlet-type 
operators in $L^2((0,\infty); \cH)$ defined by 
\begin{align}
& \bsH_{0,+} f=-f'', \no \\
&f\in \dom\big(\bsH_{0,+}\big)=\{g\in L^2((0,\infty); \cH) \,|\, g,g' \in
AC([0,R]; \cH) \text{ for all $R>0$},    \no \\
& \hspace*{6cm} f(0_+)=0, \, f''\in L^2((0,\infty); \cH)\},    \lb{4.3} \\
& \bsH_+f=-f''+Vf, \no \\
&f\in\dom(\bsH_+)=\{g\in L^2((0,\infty); \cH) \,|\, g,g' \in AC([0,R];\cH) 
\text{ for all $R>0$}, \lb{4.4} \\
& \hspace*{4.8cm} f(0_+)=0, \, (-f''+Vf) \in L^2((0,\infty); \cH)\}. \no
\end{align}

We also introduce the $\cB(\cH)$-valued Green's function of $\bsH_{0,+}$,
\begin{align}
G_{0,+}(z,x,x') &= \big(\bsH_{0,+} - z \bsI_+\big)^{-1}(x,x')=\begin{cases}  
z^{-1/2}\sin(z^{1/2}x)e^{iz^{1/2}x'} I_{\cH}, & x\leq x', \\
z^{-1/2}\sin(z^{1/2}x')e^{iz^{1/2}x} I_{\cH}, & x \geq x', 
\end{cases}    \no \\
&= \f{i}{2 z^{1/2}} \Big[e^{i z^{1/2} |x - x'|} - e^{i z^{1/2} (x + x')} \Big] I_{\cH}, 
\quad z \in \bbC \backslash \sigma(\bsH_{0,+}), \; x, x' \geq 0,     \lb{4.8}
\end{align}
with $\bsI_+$ denoting the identity operator in $L^2((0,\infty); \cH)$. Introducing the factorization
analogous to \eqref{B.12}, \eqref{B.17} (for $x \geq 0$), one verifies as in \eqref{B.13}, 
\begin{align}
& (\bsH_+-z \bsI_+)^{-1} = \big(\bsH_{0,+} - z \bsI_+\big)^{-1} \no \\
& \quad -\big(\bsH_{0,+} - z \bsI_+\big)^{-1} \bsv
\Big[I+\ol{ \bsu \big(\bsH_{0,+} - z \bsI_+\big)^{-1} \bsv}\Big]^{-1} 
\bsu \big(\bsH_{0,+} - z \bsI_+\big)^{-1}, 
\lb{4.14}  \\ 
& \hspace*{8.6cm} z\in\bbC\backslash\sigma(\bsH_+),  \no 
\end{align}
and hence also introduces the operator $\bsK_+(z)$ in $L^2((0,\infty); \cH)$ by
\begin{equation}
\bsK_+(z)=-\ol{\bsu \big(\bsH_{0,+} - z \bsI_+\big)^{-1} \bsv}, \quad 
z\in\bbC\backslash \sigma\big(\bsH_{0,+}\big),     \lb{4.15}
\end{equation}
with $\cB(\cH)$-valued integral kernel
\begin{equation}
K_+(z,x,x')=-u(x)G_{0,+}(z,x,x')v(x'), \quad \Im(z^{1/2})\geq 0, \; x,x' > 0.    \lb{4.16} 
\end{equation}

Assuming $V(x)$ to be self-adjoint for a.e.\ $x > 0$, we introduce its negative part $V_-(\cdot)$ 
(using the spectral theorem) by $V_-(\cdot) = [|V(\cdot)| - V(\cdot)]/2$ a.e.\ on $(0,\infty)$.
We also use the notation $N(\lambda; A)$, $\lambda < \inf(\sigma_{ess}(A))$ to denote the number 
of discrete eigenvalues (counting multiplicity) of the self-adjoint operator $A$ less than or equal to 
$\lambda$.  

Then the well-known Bargmann bound \cite{Ba52} on the number of negative eigenvalues for 
Dirichtlet-type half-line Schr\"odinger operators reads as follows in the current context of 
operator-valued potentials:
 
\begin{theorem} \lb{tB.5}
Assume Hypothesis \ref{hB.4} and suppose that $V(x)$ is self-adjoint in $\cH$ for a.e.\ $x > 0$. 
Then the number of negative eigenvalues of $\bsH_+$, denoted by $N(\bsH_+)$, satisfies the bound,
\begin{equation}
N(\bsH_+) \leq \int_{(0,\infty)} dx \, x \tr_{\cH}(V_-(x)).  
\end{equation}
\end{theorem} 
\begin{proof}
As usual we may replace $V(\cdot)$ consistently by $V_-(\cdot)$. The Birman--Schwinger principle 
then implies
\begin{align}
N(- \lambda; \bsH_+) & \leq 
\tr_{L^2((0,\infty); \cH)} \big(\bsv_- (\bsH_{0,+} + \lambda \bsI_+)^{-1} \bsv_-\big)    \no \\
& = \f{1}{2 \lambda^{1/2}} \int_{(0,\infty)} dx \, \big[1 - e^{-2 \lambda^{1/2}x}\big] \tr_{\cH} (V_-(x))   \no \\
& \leq \int_{(0,\infty)} dx \, x \tr_{\cH}(V_-(x)).  
\end{align} 
Here, in obvious notation, $\bsv_-$ is defined as $\bsv$ in \eqref{B.12}, \eqref{B.17}, but with 
$V(\cdot)$ replaced by $V_-(\cdot)$, and we employed the well-known inequality 
$\big[1 - e^{-r}\big] \leq r$, $r \geq 0$ (cf., e.g., \cite[4.2.29, p.\ 70]{AS72}). To complete the 
proof it suffices to let $\lambda \downarrow 0$.  
\end{proof}

This proof was kindly communicated to us by A.\ Laptev \cite{La} in the context of 
matrix-valued potentials $V(\cdot)$. The proof is clearly of a canonical nature and independent 
of the dimension of $\cH$.   

\medskip

\noindent {\bf Acknowledgments.} We are indebted to Ari Laptev for communicating Theorem 
\ref{tB.5} to us. R.N. gratefully acknowledges support from an AMS--Simons Travel Grant. 

 
\end{document}